\documentclass[12pt,reqno]{amsart}
\usepackage{amsmath,amsthm,amsfonts,amssymb,amscd,amstext}
\usepackage[ansinew]{inputenc}
\usepackage[dvips]{graphicx}
\usepackage{psfrag}
\usepackage{euscript}
\usepackage{a4wide}

\numberwithin{equation}{section}

\usepackage[active]{srcltx}
\usepackage[dvips]{color}
\usepackage[colorlinks=true,linkcolor=red,urlcolor=black]{hyperref}

\newcommand{\qand}{\quad\text{and}\quad}

\theoremstyle{plain}
\newtheorem{maintheorem}{Theorem}
\newtheorem{maincorollary}[maintheorem]{Corollary}

\newtheorem{theorem}{Theorem}[section]
\newtheorem{proposition}[theorem]{Proposition}
\newtheorem{corollary}[theorem]{Corollary}
\newtheorem{lemma}[theorem]{Lemma}

\theoremstyle{definition}
\newtheorem{remark}[theorem]{Remark}
\newtheorem{definition}{Definition}

\newtheorem{example}{Example}

\renewcommand{\angle}{\sphericalangle}

\newcommand{\RR}{{\mathbb R}}

\newcommand{\TT}{{\mathbb T}}

\newcommand{\vfi}{\varphi}

\newcommand{\la}{\lambda}

\newcommand{\vol}{\operatorname{vol}}

\renewcommand{\epsilon}{\varepsilon}
\newcommand{\dist}{\operatorname{dist}}

\newcommand{\diag}{\operatorname{diag}}

\newcommand{\gen}{\operatorname{span}}

\newcommand{\indi}{\operatorname{ind}}

\newcommand{\sing}{\mathrm{Sing}}

\newcommand{\V}{\EuScript{V}}
\newcommand{\U}{\EuScript{U}}
\newcommand{\G}{\EuScript{G}}

\newcommand{\cC}{\EuScript{C}}

\newcommand{\R}{\EuScript{R}}
\newcommand{\J}{\EuScript{J}}
\newcommand{\W}{\EuScript{W}}

\newcommand{\Mundo}{\mathfrak{X}^{1}(M)}

\begin{document}

\title%[Infinitesimal Lyapunov functions]
{Infinitesimal Lyapunov functions for singular flows}

%%%%%%%%%%%%%%%%%%%%%%%%%%%%%%%%%%%%%%%%%%%%%%%%%%%%%%%%%%%%%%%%%%%

\thanks{
  V.A. was partially supported by CAPES, CNPq,
  PRONEX-Dyn.Syst. and FAPERJ (Brazil). L.S. was supported
  by a CNPq doctoral scholarship and is presently supported
  by a INCTMat-CAPES post-doctoral scholarship.
}

\subjclass{Primary: 37D30; Secondary: 37D25.}
\renewcommand{\subjclassname}{\textup{2000} Mathematics
  Subject Classification}
\keywords{Dominated splitting,
  partial hyperbolicity, sectional hyperbolicity,
  infinitesimal Lyapunov function.}

%%%%%%%%%%%%%%%%%%%%%%%%%%%%%%%%%%%%%%%%%%%%%%%%%%%%%%%%%%%%%%%%%%%

\author{Vitor Araujo}
\address[V.A.]{Universidade Federal da Bahia,
Instituto de Matem\'atica\\
Av. Adhemar de Barros, S/N , Ondina,
40170-110 - Salvador-BA-Brazil}
\email{vitor.d.araujo@ufba.br}

\author{Luciana Salgado} \address[L.S.]{ Instituto de
  Matem\'atica Pura e Aplicada - Estrada Dona Castorina,
  110, Jardim Bot\^anico, 22460-320 Rio de Janeiro, Brazil }
\email{lsalgado@impa.br}

%%%%%%%%%%%%%%%%%%%%%%%%%%%%%%%%%%%%%%%%%%%%%%%%%%%%%%%%%%%%%%%%%%%
\begin{abstract}
  We present an extension of the notion of infinitesimal
  Lyapunov function to singular flows, and from this
  technique we deduce a characterization of partial/sectional
  hyperbolic sets. In absence of singularities, we can also
  characterize uniform hyperbolicity.

These conditions can be expressed using the space
  derivative $DX$ of the vector field $X$ together with a
  field of infinitesimal Lyapunov functions only, and are
  reduced to checking that a certain symmetric operator is
  positive definite at the tangent space of every point of
  the trapping region.
\end{abstract}

%%%%%%%%%%%%%%%%%%%%%%%%%%%%%%%%%%%%%%%%%%%%%%%%%%%%%%%%%%%%%%%%%%%

\date{\today}

\maketitle
%\tableofcontents

\section{Introduction}
\label{sec:introd-statem-result}

The hyperbolic theory of dynamical systems is now almost a
classical subject in mathematics and one of the main
paradigms in dynamics. Developed in the 1960s and 1970s
after the work of Smale, Sinai, Ruelle, Bowen
\cite{Sm67,Si68,Bo75,BR75}, among many others, this theory
deals with compact invariant sets $\Lambda$ for
diffeomorphisms and flows of closed finite-dimensional
manifolds having a hyperbolic splitting of the tangent
space. That is, $T_\Lambda M= E^s\oplus E^X \oplus E^u$ is a
continuous splitting of the tangent bundle over $\Lambda$,
where $E^X$ is the flow direction, the
subbundles are invariant under the derivative $DX_t$ of the
flow $X_t$
\begin{align*}
  DX_t\cdot E^*_x=E^*_{X_t(x)},\quad  x\in\Lambda, \quad t\in\RR,\quad *=s,X,u;
\end{align*}
$E^s$ is uniformly contracted by $DX_t$ and $E^u$ is
likewise expanded: there are $K,\lambda>0$ so that
\begin{align*}
  \|DX_t\mid E^s_x\|\le K e^{-\lambda t},
  \quad
  \|(DX_t \mid E^u_x)^{-1}\|\le K e^{-\lambda t},
  \quad x\in\Lambda, \quad t\in\RR.
\end{align*}
Very strong properties can be deduced from the existence of
such hyperbolic structure; see for
instance~\cite{Bo75,BR75,Sh87,KH95,robinson2004}.

More recently, extensions of this theory based on weaker
notions of hyperbolicity, like the notions of dominated
splitting, partial hyperbolicity, volume hyperbolicity and
singular hyperbolicity (for three-dimensional flows) have
been developed to encompass larger classes of systems beyond
the uniformly hyperbolic ones; see~\cite{BDV2004} and
specifically~\cite{viana2000i,AraPac2010} for singular
hyperbolicity and Lorenz-like attractors.

One of the technical difficulties in this theory is to
actually prove the existence of a hyperbolic structure, even
in its weaker forms.  We mention that Malkus
showed that the Lorenz equations,
presented in~\cite{Lo63}, are the equations of motion of a
waterwheel, which was built at MIT in the 1970s and helped
to convince the skeptical physicists of the reality of
chaos; see \cite[Section 9.1]{strogatz94}.  Only around the
year 2000 was it established by Tucker in \cite{Tu99,Tu2}
that the Lorenz system of equations, with the parameters
indicated by Lorenz, does indeed have a chaotic strange
attractor.  This proof is a computer assisted proof which
works for a specific choice of parameters, and has not been
improved to this day. More recently, Hunt and Mackay in
\cite{HuntMack03} have shown that the behavior of a certain
physical system, for a specific choice of parameters which
can be fixed in a concrete laboratory setup, is modeled by
an Anosov flow.

The most usual and geometric way to prove hyperbolicity is
to use a field of cones. This idea goes as far back as the
beginning of the hyperbolic theory; see Alekseev
\cite{Aleks68,Aleks68a,Aleks69}. Given a continuous, 
splitting $T_\Lambda M=E\oplus F$ (not necessarily invariant
with respect to a flow $X_t$) of the tangent space over
an invariant subset $\Lambda$, a field of cones of size $a>0$
centered around $F$ is defined by
\begin{align*}
  C_a(x):=\{\vec0\}\cup\{(u,v)\in E_x\times F_x
  : \| u\| \le a\|v\|\}, \quad x\in\Lambda.
\end{align*}
Let us assume that there exists $\lambda\in(0,1)$ such that for
all $x\in\Lambda$ and every negative $t$
\begin{enumerate}
\item $\overline{DX_t\cdot C_a(x)}\subset C_a(X_t(x))$ (the
  overline denotes closure in $T_{X_t(x)}M$);
\item $\|DX_t\cdot w\| \le \lambda \|w\|$ for each $v\in C_a(x)$.
\end{enumerate}
Then there exists an invariant bundle $E^s$ contained in the
cone field $C_a$ over $\Lambda$ whose vectors are uniformly
contracted. The complementary cone field satisfies the
analogous to the first item above for positive $t$. This
ensures the existence of a partially hyperbolic splitting
over $\Lambda$.

We present a simple extension of the notion of infinitesimal
Lyapunov function, from \cite{BurnKatok94}, to singular
flows, and show how this technique provides a new
characterization of partially hyperbolic structures for
invariant sets for flows, and also of singular and sectional
hyperbolicity. In the absence of singularities, we can also
rephrase uniform hyperbolicity with the language of
infinitesimal Lyapunov functions.

This technique is not new. Lewowicz used it in his study of
expansive homeomorphisms~\cite{lewow89} and Wojtkowski
adapted it for the study of Lyapunov exponents in
\cite{Wojtk85}.  Using infinitesimal Lyapunov functions,
Wojtkowski was able to show that the second item above is
superfluous: the geometric condition expressed in the first
item is actually enough to conclude uniform contraction.

Workers using these techniques, like
Lewowicz~\cite{lewow80}, Markarian~\cite{Mrkr88},
Wojtkowski~\cite{Wojtk00a}, Burns-Katok~\cite{BurnKatok94},
have only considered either dynamical systems given by maps
or by flows without singularities. In this last case, the
authors deal with the Linear Poincar\'e Flow on the normal
bundle to the flow direction.

We adapt ideas introduced first by Lewowicz
in~\cite{lewow80}, and developed by several other authors in
different contexts, to the setting of vector bundle
automorphisms over flows with singularities; see
also~\cite{Wojtk01,Wojtk00,Mrkr88} for other known
applications of this technique to billiards and symplectic
flows, and also~\cite{Pota79,Pota60} for a general theory of
$\J$-monotonous linear transformations. Recently in
\cite{BaDrVal12} a general framework was established
relating Lyapunov exponents for invariant measures of
maps and flows with eventually strict Lyapunov functions.

We improve on these results by showing, roughly, that the
condition on item (1) above on a trapping region for a flow
implies partial hyperbolicity, even when singularities are
present. This also provides a way to define uniform
hyperbolicity on a compact invariant set for a smooth flow
generated by a vector field $X$, using only $X$ and $DX$
together with a family of Lyapunov functions.

We then provide an extra necessary and sufficient condition
ensuring that the complementary cone, containing invariant
subbundle $E^c$, which contains the flow direction, is such
that the area form along any two-dimensional subspace of
$E^c$ is uniformly expanded by the action of the tangent
cocycle $DX_t$ of the flow $X_t$ (this property is today
known as sectional-hyperbolicity; see \cite{MeMor06}).

Moreover, these conditions can be expressed using the vector
field $X$ and its space derivative $DX$ together with an
infinitesimal Lyapunov function, and are reduced to checking
that a certain symmetric operator is positive definite on
all points of the trapping region. While we usually define
hyperbolicity by using the differential $Df$ of a
diffeomorphism $f$, or the cocycle $(DX_t)_{t\in\RR}$
associated to the continuous one-parameter group
$(X_t)_{t\in\RR}$, we show how to express partial
hyperbolicity using only the interplay between the
infinitesimal generator $X$ of the group $X_t$, its
derivative $DX$ and the infinitesimal Lyapunov function.
Since in many situations dealing with mathematical models
from the physical, engineering or social sciences, it is the
vector field that is given and not the flow, we expect that
the theory here presented to be useful to develop simpler
algorithms to check hyperbolicity.

\subsection{Preliminary definitions}
\label{sec:prelim-definit}

Before the main statements we collect some definitions in
order to state the main results.

Let $M$ be a connected compact finite $n$-dimensional
manifold, $n \geq 3$, with or without boundary, together
with a flow $X_t : M \to M, t \in \mathbb{R}$ generated by a
$C^1$ vector field $X: M \to TM$, such that $X$ is inwardly
transverse to the boundary $\partial M$, if $\partial
M\neq\emptyset$.

An \emph{invariant set} $\Lambda$ for the flow of $X$ is a
subset of $M$ which satisfies $X_t(\Lambda)=\Lambda$ for all
$t\in\RR$.  The \emph{maximal invariant set of the flow} is
$M(X):= \cap_{t \geq 0} X_t(M)$, which is clearly a compact
invariant set.

A \emph{trapping region} $U$ for a flow $X_t$ is an
open subset of the manifold $M$ which satisfies: $X_t(U)$ is
contained in $U$ for all $t>0$; and there exists $T>0$ such
that $\overline{X_t(U)} $ is contained in the interior of
$U$ for all $t>T$.

A \emph{singularity} for the vector field $X$
is a point $\sigma\in M$ such that $X(\sigma)=\vec0$ or,
equivalently, $X_t(\sigma)=\sigma$ for all $t \in \RR$. The
set formed by singularities is the \emph{singular set of
  $X$} denoted $\sing(X)$.  We say that a \emph{singularity is
  hyperbolic} if the eigenvalues of the derivative
$DX(\sigma)$ of the vector field at the singularity $\sigma$
have nonzero real part.
%A Lorenz-like singularity $\sigma$ for a $C^1$
%  flow $\{X_t\}_{t\in\R}$ on a $3$-manifold $M$ is
%  a hyperbolic singularity of saddle-type such
%  that the eigenvalues of $DX(\sigma)$ are real and satisfy
%  \begin{align*}
%    \lambda_2 < \lambda_3 < 0 < -\lambda_3 < \lambda_1.
%  \end{align*}
\begin{definition}\label{def1}
  %\cite[Definition 2.6]{MeMor06}
  A \emph{dominated splitting} over a compact invariant set $\Lambda$ of $X$
  is a continuous $DX_t$-invariant splitting $T_{\Lambda}M =
  E \oplus F$ with $E_x \neq \{0\}$, $F_x \neq \{0\}$ for
  every $x \in \Lambda$ and such that there are positive
  constants $K, \lambda$ satisfying
  \begin{align}\label{eq:def-dom-split}
    \|DX_t|_{E_x}\|\cdot\|DX_{-t}|_{F_{X_t(x)}}\|<Ke^{-\la
      t}, \ \textrm{for all} \ x \in \Lambda, \ \textrm{and
      all} \,\,t> 0.
  \end{align}
\end{definition}

A compact invariant set $\Lambda$ is said to be
\emph{partially hyperbolic} if it exhibits a dominated
splitting $T_{\Lambda}M = E \oplus F$ such that subbundle
$E$ is uniformly contracted. In this case $F$ is the
\emph{central subbundle} of $\Lambda$.

A compact invariant set $\Lambda$ is said to be
\emph{singular hyperbolic} if it is partially hyperbolic and
the action of the tangent cocycle expands volume along the
central subbundle, i.e.,
\begin{align}\label{eq:def-vol-exp}
      \vert \det (DX_t\vert_{F_x}) \vert > C e^{\la t},
      \forall t>0, \ \forall \ x \in \Lambda.
    \end{align}

\begin{definition}\label{def:sec-exp}
  We say that a $DX_t$-invariant subbundle $F \subset
  T_{\Lambda}M$ is a \emph{sectionally expanding} subbundle
  if $\dim F_x \geq 2$ is constant for $x\in\Lambda$
  and there are positive constants $C , \lambda$ such that for every $x
    \in \Lambda$ and every two-dimensional linear subspace
    $L_x \subset F_x$ one has
    \begin{align}\label{eq:def-sec-exp}
      \vert \det (DX_t \vert_{L_x})\vert > C e^{\la t},
      \forall t>0.
    \end{align}
  \end{definition}

And, finally, the definition of sectional-hyperbolicity.

\begin{definition}\label{sechypset} \cite[Definition
  2.7]{MeMor06} A \emph{sectional hyperbolic set} is a
  partially hyperbolic set whose singularities are
  hyperbolic and the central subbundle is sectionally
  expanding.
\end{definition}

We note that a sectional hyperbolic set always is
singular hyperbolic, however the reverse is only true
in dimension three, not in higher dimensions; see for
instance Metzger-Morales \cite{MeMor06} and Zu-Gan-Wen
\cite{ZGW08}.

\begin{remark}
  \label{rmk:sec-exp-discrete}
  The properties of sectional hyperbolicity can be expressed in the following
  equivalent way; see \cite{AraPac2010}.  There exists
  $T>0$ such that
  \begin{itemize}
  \item $\|DX_T\vert_{E_x}\|<\frac12$ for all $x\in\Lambda$
    (uniform contraction); and
  \item $|\det (DX_T\vert_{ F_x})|> 2$ for all $x\in\Lambda$
    and each $2$-subspace $F_x$ of $E^c_x$ ($2$-sectional expansion).
  \end{itemize}
\end{remark}

We say that a compact invariant set $\Lambda$ is a
\emph{volume hyperbolic} set if it has a dominated splitting
$E\oplus F$ such that the volume along its subbundles is
uniformly contracted (along $E$) and expanded (along $F$) by
the action of the tangent cocyle. If the whole manifold $M$
is a volume hyperbolic set for a flow $X_t$, then we say
that $X_t$ is a volume-hyperbolic flow.

We recall that a flow $X_t$ is said to be \emph{Anosov} if
the whole manifold $M$ is a hyperbolic set for the flow.
Based on this definition, we say that $X_t$ is a
\emph{sectional Anosov flow} if the maximal invariant set
$M(X)$ is a sectional-hyperbolic set for the flow.

From now on, we consider $M$ a connected compact finite
dimensional riemannian manifold, $U \subset M$ a trapping
region, $\Lambda(U)=\Lambda_X(U):= \cap_{t>0}\overline
{X_t(U)}$ the maximal positive invariant subset in $U$ for
the vector field $X$ and $E_U$ a finite dimensional vector
bundle over $U$.  We also assume that all singularities of
$X$ in $U$ (if they exist) are hyperbolic.

\subsubsection{Fields of quadratic forms; positive and
  negative cones}
\label{sec:fields-quadrat-forms}

Let $E_U$ be a finite dimensional vector bundle with base
$U$ and $\J:E_U\to\RR$ be a continuous field of quadratic
forms $\J_x:E_x\to\RR$ which are non-degenerate and have
index $0<q<\dim(E)=n$.  The index $q$ of $\J$ means that the
maximal dimension of subspaces of non-positive vectors is
$q$.

We also assume that $(\J_x)_{x\in U}$ is continuously
differentiable along the flow.  The continuity assumption on
$\J$ means that for every continuous section $Z$ of $E_U$
the map $U\ni x\mapsto \J(Z(x))\in\RR$ is
continuous. The $C^1$ assumption on $\J$ along the flow
means that the map $\RR\ni t\mapsto \J_{X_t(x)} (Z(X_t(x)))\in \RR$
is continuously differentiable for all $x\in U$ and each
$C^1$ section $Z$ of $E_U$.

We let $\cC_\pm=\{C_\pm(x)\}_{x\in U}$ be the field of
positive and negative cones associated to $\J$
\begin{align*}
  C_\pm(x):=\{0\}\cup\{v\in E_x: \pm\J_x(v)>0\}  \quad x\in U
\end{align*}
and also let $\cC_0=\{C_0(x)\}_{x\in U}$ be the corresponding
field of zero vectors $C_0(x)=\J_x^{-1}(\{0\})$ for all
$x\in U$.
% In the adapted coordinates obtained above we have
% \begin{align*}
%   C_0(x)=\{v=\sum_{i}\alpha_iv_i\in E_x :
%   \sum_{j=q+1}^n \alpha_j^2 = \sum_{i=1}^q
%   \alpha_i^2\}
% \end{align*}
% is the set of \emph{extreme points} of $C_\pm(x)$.

\subsubsection{Linear multiplicative cocycles over flows}
\label{sec:linear-multipl-cocyc}

Let $A:E\times\RR\to E$ be a smooth map given by a collection of linear
bijections
\begin{align*}
  A_t(x): E_x\to E_{X_t(x)}, \quad x\in M, t\in\RR,
\end{align*}
where $M$ is the base space of the finite dimensional vector
bundle $E$, satisfying the cocycle property
\begin{align*}
  A_0(x)=Id, \quad A_{t+s}(x)=A_t(X_s(x))\circ A_s(x), \quad x\in M, t,s\in\RR,
\end{align*}
with $\{X_t\}_{t\in\RR}$ a smooth flow over $M$.  We note
that for each fixed $t>0$ the map $A_t: E\to E, v_x\in E_x
\mapsto A_t(x)\cdot v_x\in E_{X_t(x)}$ is an automorphism of
the vector bundle $E$.

The natural example of a linear multiplicative cocycle over
a smooth flow $X_t$ on a manifold is the derivative cocycle
$A_t(x)=DX_t(x)$ on the tangent bundle $TM$ of a finite
dimensional compact manifold $M$.

The following definitions are fundamental to state our
results.

\begin{definition}
\label{def:J-separated}
Given a continuous field of non-degenerate quadratic forms
$\J$ with constant index on the trapping region $U$ for the
flow $X_t$, we say that the cocycle $A_t(x)$ over $X$ is
\begin{itemize}
\item $\J$-\emph{separated} if $A_t(x)(C_+(x))\subset
  C_+(X_t(x))$, for all $t>0$ and $x\in U$;
\item \emph{strictly $\J$-separated} if $A_t(x)(C_+(x)\cup
  C_0(x))\subset C_+(X_t(x))$, for all $t>0$ and $x\in U$;
\item $\J$-\emph{monotone} if $\J_{X_t(x)}(A_t(x)v)\ge \J_x(v)$, for each $v\in
  T_xM\setminus\{\vec0\}$ and $t>0$;
\item \emph{strictly $\J$-monotone} if $\partial_t\big(\J_{X_t(x)}(A_t(x)v)\big)\mid_{t=0}>0$,
  for all $v\in T_xM\setminus\{0\}$, $t>0$ and $x\in U$;
\item $\J$-\emph{isometry} if $\J_{X_t(x)}(A_t(x)v) = \J_x(v)$, for each $v\in T_xM$ and $x\in U$.
\end{itemize}
\end{definition}
Thus, $\J$-separation corresponds to simple cone invariance
and strict $\J$-separation corresponds to strict cone
invariance under the action of $A_t(x)$.

We say that the flow $X_t$ is (strictly)
$\J$-\emph{separated} on $U$ if $DX_t(x)$ is (strictly)
$\J$-\emph{separated} on $T_UM$.
Analogously, the flow of $X$ on $U$ is (strictly) $\J$-\emph{monotone} if
$DX_t(x)$ is (strictly) $\J$-\emph{monotone}.

\begin{remark}\label{rmk:J-separated-C-}
  If a flow is strictly $\J$-separated, then for $v\in T_xM$
  such that $\J_x(v)\le0$ we have
  $\J_{X_{-t}(x)}(DX_{-t}(v))<0$ for all $t>0$ and $x$ such
  that $X_{-s}(x)\in U$ for every $s\in[-t,0]$. Indeed,
  otherwise $\J_{X_{-t}(x)}(DX_{-t}(v))\ge0$ would imply
  $\J_x(v)=\J_x\big(DX_t(DX_{-t}(v))\big)>0$, contradicting
  the assumption that $v$ was a non-positive vector.

  This means that a flow $X_t$ is strictly
    $\J$-separated if, and only if, its time reversal
    $X_{-t}$ is strictly $(-\J)$-separated. % The same remark
  % applies in the case of a cocycle.
\end{remark}

A vector field $X$ is $\J$-\emph{non-negative} on $U$ if
$\J(X(x))\ge0$ for all $x\in U$, and
$\J$-\emph{non-positive} on $U$ if $\J(X(x))\leq 0$ for all
$x\in U$. When the quadratic form used in the context is
clear, we will simply say that $X$ is non-negative or
non-positive.

\subsection{Statement of the results}
\label{sec:j-separat-flows}

We say that a compact invariant subset $\Lambda$ is
\emph{non-trivial} if
\begin{itemize}
\item either $\Lambda$ does not contain singularities;
\item or $\Lambda$ contains at most finitely
many singularities, $\Lambda$ contains some
regular orbit and is connected.
\end{itemize}

Our main result is the following.

\begin{maintheorem}
  \label{mthm:Jseparated-parthyp}
  A non-trivial attracting set $\Lambda$ of a trapping
  region $U$ is partially hyperbolic for a flow $X_t$ if,
  and only if, there is a $C^1$ field of non-degenerate
  quadratic forms $\J$ with constant index, equal to the
  dimension of the stable subspace of $\Lambda$, such that
  $X_t$ is non-negative strictly $\J$-separated on $U$.
\end{maintheorem}

This is a direct consequence of a corresponding result for
linear multiplicative cocycles over vector bundles which we
state in Theorem~\ref{thm:strict-J-separated-cocycle} and
prove in Section~\ref{sec:j-separat-cocycl}.

We obtain a criterion for partial and uniform hyperbolicity
which extends the one given by Lewowicz, in \cite{lewow80},
and Wojtkowski, in \cite{Wojtk01}. The condition of strict
$\J$-separation can be expressed only using the vector field
$X$ and its spatial derivative $DX$, as follows.

\begin{proposition}
  \label{pr:J-separated}
  A $\J$-non-negative vector field $X$ on $U$ is (strictly)
  $\J$-separated if, and only if, there exists a compatible
  field of forms $\J_0$ and there exists a function
  $\delta:U\to\RR$ such that the operator $\tilde J_{0,x}:=
  J_0\cdot DX(x)+DX(x)^*\cdot J_0$ satisfies
  \begin{align*}
    \tilde J_{0,x}-\delta(x)J_0 \quad\text{is positive
      (definite) semidefinite}, \quad x\in U,
  \end{align*}
  where $DX(x)^*$ is the adjoint of $DX(x)$ with respect to
  the adapted inner product.
\end{proposition}

In the statement above, we say that a field of
quadratic forms $\J_0$ on $U$ is \emph{compatible} to $\J$, and we
write $\J\sim\J_0$, if there exists $C>1$ satisfying for
$x\in\Lambda$
\begin{align*}
  \frac1C\cdot\J_0(v)\le \J(v) \le C\cdot\J_0(v),
\quad v\in E_x\cup F_x,
\end{align*}
where $E\oplus F$ is a $DX_t$-invariant splitting of
$T_\Lambda M$.

Again this is a consequence of a corresponding result for
linear multiplicative cocycles where $DX$ is replaced by the
infinitesimal generator
\begin{align*}%\label{eq:infinitesimal-gen}
  D(x):=\lim_{t\to0}\frac{A_t(x)-Id}t
\end{align*}
of the cocycle $A_t(x)$.

As a consequence of Theorem \ref{mthm:Jseparated-parthyp}, we
characterize hyperbolic maximal invariant subsets in
trapping regions without singularities as follows.
We recall that the index of a (partially) hyperbolic set is
the dimension of the uniformly contracted subbundle of its
tangent bundle.

\begin{maincorollary}\label{mcor:J-separation-hyperbolicity}
  The maximal invariant subset $\Lambda$ of $U$ is a
  hyperbolic set for $X$ of index $s$ if, and only if, there
  exist $\J,\G$ smooth fields of non-degenerate quadratic
  forms on $U$ with constant index $s$ and $n-s-1$,
  respectively, where $s<n-2$ and $n = dim(M)$, such that
  $X_t$ is strictly $\J$-separated non-negative on $U$ with
  respect to $\J$, $X_t$ is strictly $\G$-separated
  non-positive with respect to $\G$, and there are no
  singularities of $X$ in $U$.
\end{maincorollary}

\subsubsection{Incompressible vector fields}
To state the next result, we recall that a vector field is
said to be \emph{incompressible} if its flow has null
divergence, i.e., it is volume-preserving on $M$.

In this particular case, we have the following easy
corollary of Theorem~\ref{mthm:Jseparated-parthyp}, since a
partially hyperbolic flow in a compact
manifold %whose stable direction has codimension $2$
must expand volume along the central direction. Moreover, if
the stable direction has codimension $2$, the central
direction expands area.
\begin{maincorollary}
  \label{mcor:incompress-anosov}
  Let $X$ be a $C^1$ incompressible vector field on a
  compact finite dimensional manifold $M$ which is
  non-negative and strictly $\J$-separated for a field 
  of non-degenerate and indefinite quadratic forms $\J$ with
  index $\indi(\J)= \dim(M)-2$.  Then $X_t$ is an Anosov
  flow.
\end{maincorollary}
Indeed, the results of Doering~\cite{Do87} and
Morales-Pacifico-Pujals \cite{MPP99}, in dimension three,
Vivier~\cite{Viv03} and Li-Gan-Wen~\cite{LGW05}, in
higher dimensions, ensure that there are no singularities in
the interior of a sectional-hyperbolic set, and so this set
is hyperbolic; see Section~\ref{sec:uniform-hyperb}.

To present the results about sectional-hyperbolicity, we
need some more definitions.

\subsubsection{$\J$-monotonous Linear Poincar\'e Flow}
\label{sec:j-monotonous-linear}

We apply these notions to the linear Poincar\'e flow defined
on regular orbits of $X_t$ as follows.

We assume that the vector field $X$ is non-negative on $U$.
Then, the span $E^X_x$ of $X(x)\neq\vec0$ is a
$\J$-non-degenerate subspace.  According to item (1) of
Proposition~\ref{pr:propbilinear}, this means that
$T_xM=E_x^X\oplus N_x$, where $N_x$ is the pseudo-orthogonal
complement of $E^X_x$ with respect to the bilinear form
$\J$, and $N_x$ is also non-degenerate. Moreover, by the
definition, the index of $\J$ restricted to $N_x$ is the
same as the index of $\J$. Thus, we can define on $N_x$ the
cones of positive and negative vectors, respectively,
$N_x^+$ and $N_x^-$, just like before.

Now we define the Linear Poincar\'e Flow $P^{\, t}$ of $X_t$
along the orbit of $x$, by projecting $DX_t$ orthogonally
(with respect to $\J$) over $N_{X_t(x)}$ for each $t\in\RR$:
\begin{align*}
  P^{\, t} v := \Pi_{X_t(x)}DX_t v ,
  \quad
  v\in T_x M, t\in\RR, X(x)\neq\vec0,
\end{align*}
where $\Pi_{X_t(x)}:T_{X_t(x)}M\to N_{X_t(x)}$ is the
projection on $N_{X_t(x)}$ parallel to $X(X_t(x))$.  We
remark that the definition of $\Pi_x$ depends on $X(x)$ and
$\J_X$ only. The linear Poincar\'e flow $P^{\,t}$ is a linear
multiplicative cocycle over $X_t$ on the set $U$ with the
exclusion of the singularities of $X$.

In this setting we can say that the linear Poincar\'e flow is
(strictly) $\J$-separated and (strictly) $\J$-monotonous
using the non-degenerate bilinear form $\J$ restricted to
$N_x$ for a regular $x\in U$. More precisely: $P^t$ is
$\J$-monotonous if $\partial_t\J(P^tv)\mid_{t=0}\ge0$, for
each $x\in U, v\in T_xM\setminus\{\vec0\}$ and $t>0$, and
strictly $\J$-monotonous if $\partial_t\J(P^tv)\mid_{t=0}>0$,
for all $v\in T_xM\setminus\{\vec0\}$, $t>0$ and $x\in
U$.

\begin{maintheorem}\label{mthm:2-sec-exp-J-monot}
  Let $\Lambda$ be a non-trivial attracting set
  $\Lambda$ of $U$ which is contained in the
  non-wandering set $\Omega(X)$.  Then $\Lambda$ is
  sectional hyperbolic for $X_t$ if, and only if, there
  is a $C^1$ field $\J$ of non-degenerate quadratic
  forms with constant index, equal to the dimension of
  the stable subspace of $\Lambda$, such that $X_t$ is
  a non-negative strictly $\J$-separated flow on $U$,
  whose singularities are sectionally hyperbolic with
  index $\indi(\sigma)\ge\indi(\Lambda)$, and for each
  compact invariant subset $\Gamma$ of $\Lambda$
  without singularities there exists a field of
  quadratic forms $\J_0$ equivalent to $\J$ so that the linear
  Poincar\'e flow is strictly $\J_0$-monotonous on
  $\Gamma$.
\end{maintheorem}

As usual, we say that $q\in M$ is {\em non-wandering}
for $X$ if for every $T>0$ and every neighborhood $W$
of $q$ there is $t>T$ such that $X_t(W)\cap W\neq
\emptyset$.  The set of non-wandering points of $X$ is
denoted by $\Omega(X)$.  A singularity $\sigma$ is
\emph{sectionally hyperbolic with index
  $\indi(\sigma)$} if $\sigma$ is a hyperbolic singularity
with stable direction $E^s_\sigma$ having dimension $\indi(\sigma)$
and a central direction $E^c_\sigma$ such that
$T_\sigma M=E^s_\sigma\oplus E^c_\sigma$ is a
$DX_t(\sigma)$-invariant splitting, $E^s_\sigma$ is
uniformly contracted and $E^c_\sigma$ is sectionally
expanded by the action of $DX_t(\sigma)$.

As before, the condition of $\J$-monotonicity for the Linear
Poincar\'e Flow can be expressed using only the vector field
$X$ and its space derivative $DX$ as follows.

\begin{proposition}
  \label{pr:P-J-monotonous}
  A $\J$-non-negative vector field $X$ on a forward
  invariant region $U$ has a Linear Poincar\'e Flow which is (strictly)
  $\J$-monotone if, and only if, the operator $\hat J_x:=
  DX(x)^*\cdot \Pi_x^* J \Pi_x + \Pi_x^* J \Pi_x\cdot DX(x)$
  is a (positive) non-negative self-adjoint operator, that
  is, all eigenvalues are (positive) non-negative, for each
  $x\in U$ such that $X(x)\neq\vec0$.
\end{proposition}

The conditions above are again consequence of the corresponding
results for linear multiplicative cocycles over flows, as
explained in Section~\ref{sec:approach-via-linear}.

%%%%%%%%%%%%%%%%%%%%%%%%%%%%%%%%%%%%%%%%%%%%%%%%%%%%%%

\subsection{Examples}
\label{sec:exampl}

With the equivalence provided by
Theorems~\ref{mthm:Jseparated-parthyp}
and~\ref{mthm:2-sec-exp-J-monot} and
Corollaries~\ref{mcor:J-separation-hyperbolicity} and
\ref{mcor:incompress-anosov}, we have plenty of examples
illustrating our results.

\begin{example}\label{ex:classical}
We can consider
\begin{itemize}
\item the classical examples of uniformly hyperbolic
  attractors for $C^1$ flows, in any dimension greater or
  equal to $3$; see e.g. \cite{BR75}.
\item the classical Lorenz attractor from the Lorenz ODE
  system and the geometrical Lorenz attractors; see
  e.g.~\cite{Lo63,Tu99,viana2000i}.
\item singular-hyperbolic (or Lorenz-like) attractors and
  attracting sets in three dimensions; see
  e.g.~\cite{MPP04,Morales07,AraPac2010}.
\item contracting Lorenz (Rovella) attractors; see
  e.g. \cite{Ro93,MetzMor06}.
\item sectional-hyperbolic attractors for dimensions higher
  than three; see e.g~\cite{BPV97,BDV2004,MeMor06}.
\item multidimensional Rovella-like attractors; see
  \cite{ACPP11}.
\end{itemize}
\end{example}

The following examples illustrate the fact that the change
of coordinates to adapt the quadratic forms as explained in
Section~\ref{sec:propert-quadrat-forms} is important in
applications.

\begin{example}
  \label{ex:suspension}
  Given a diffeomorphism $f:\TT^2\to \TT^2$ of the
  $2$-torus, let $X_t:M\to M$ be a suspension flow with roof
  function $r:M\to[r_0,r_1]$ over the base transformation
  $f$, where $0<r_0<r_1$ are fixed, as follows.

  We define $M:=\{ (x,y)\in \TT^2\times[0,+\infty): 0\le y <
  r(x) \}$.  For $x=x_0\in \TT^2$ we denote by $x_n$ the
  $n$th iterate $f^n(x_0)$ for $n\ge0$ and by $S_n \vfi(x_0)
  = S_n^f \vfi(x_0) = \sum_{j=0}^{n-1} \vfi( x_j )$ the
  $n$th-ergodic sum, for $n\ge1$ and for any given real
  function $\vfi:\TT^2\to\RR$ in what follows. Then for each
  pair $(x_0,s_0)\in M$ and $t>0$ there exists a unique
  $n\ge1$ such that $S_n r(x_0) \le s_0+ t < S_{n+1} r(x_0)$
  and we define
\begin{align*}
   X_t(x_0,s_0) = \big(x_n,s_0+t-S_n r(x_0)\big).
\end{align*}
We note that the vector field corresponding to this
suspension flow is the constant vector field $X=(0,1)$. We
observe that the space $M$ becomes a compact manifold if we
identify $(x,r(x))$ with $(f(x),0)$; see e.g.  \cite{PM82}.

Hence, if we are given a field of quadratic forms $\J$ on
$M$ and do not change coordinates accordingly, we obtain
$DX\equiv0$ and so the relation provided by
Proposition~\ref{pr:J-separated} will not be fulfilled,
because $\tilde J_x -\delta(x) J = -\delta(x) J$ is not
positive definite for any choice of $\delta$.
\end{example}

\begin{remark}\label{rmk:delta-null}
  Moreover, if a flow $X_t$ is such that the infinitesimal
  generator $X$ is constant in the ambient space is
  $\J$-separated, then strict $\J$-separation implies that
  $-\delta(x) J$ is positive definite for all $x\in U$, and
  so $\delta$ is the null function on the trapping region.
\end{remark}

\begin{example}
  \label{ex:J-separated}
  Now consider the same example as above but now $f$ is an
  Anosov diffeomorphism of $\TT^2$ with the hyperbolic
  splitting $E^s\oplus E^u$ defined at every point. Then the
  semiflow will be partially hyperbolic with splitting
  $E^s\oplus (E^X\oplus E^u)$ where $E^X$ is the
  one-dimensional bundle spanned by the flow direction:
  $E^X_{(x,s)}=\RR\cdot X(x,s), (x,s)\in M$.

  Hence, Theorem~\ref{mthm:Jseparated-parthyp} ensures the
  existence of a field $\J$ of quadratic forms such that
  $X^t$ is strictly $\J$-separated.

  Comparing with the observation at the end of Example
  \ref{ex:suspension}, this demands a change of
  coordinates and, in those coordinates, the vector field
  $X$ will no longer be a constant vector field.
\end{example}

\begin{example}
  \label{ex:no-domination}
  Now we present a suspension flow whose base map has a
  dominated splitting but the flow does not admit any
  dominated splitting.

  Let $f:\TT^4\times\TT^4$ be the diffeomorphism described
  in \cite{BoV00} which admits a continuous dominated
  splitting $E^{cs}\oplus E^{cu}$ on $\TT^4$, but does not
  admit any hyperbolic (uniformly contracting or expanding)
  sub-bundle.
  \begin{figure}[htpb]
    \includegraphics[height=3cm]{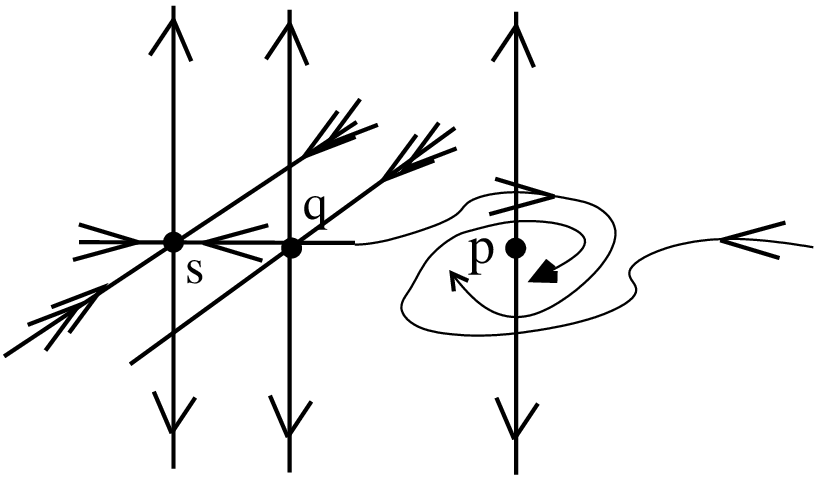}
    \caption{Saddles with real and complex eigenvalues.}
    \label{fig:nonhypcomplex}
  \end{figure}
  There are hyperbolic fixed points of $f$ satisfying, see
  Figure~\ref{fig:nonhypcomplex}:
  \begin{itemize}
  \item $\dim E^u(p)=2=\dim E^s(p)$ and there exists no
    invariant one-dimensional sub-bundle of $E^u(p)$;
  \item $\dim E^u(\widetilde p)=2=\dim E^s(\widetilde p)$ and there exists no
    invariant one-dimensional sub-bundle of $E^s(\widetilde
    p)$;
  \item $\dim E^s(\widetilde q)=3$ and $\dim E^u(q)=3$.
  \end{itemize}
  Hence, the suspension semiflow of $f$ with constant roof
  function $1$ does not admit any dominated splitting. In
  fact, the natural invariant splitting $E^{cs}\oplus
  E^X\oplus E^{cu}$ is the continuous invariant splitting
  over $\TT^4\times [0,1]$ with bundles of least dimension,
  and is not dominated since at the point $p$ the flow
  direction $E^X(p)$ dominates the $E^{cs}(p)=E^s(p)$
  direction, but at the point $q$ this domination is
  impossible.
\end{example}

We now present simple cases of partial hyperbolicity/strict
separation and hyperbolicity/strict monotonicity,
obtaining explicitly the function $\delta$.

\begin{example}\label{ex:Lorenz-like-J}
  Let us consider a hyperbolic saddle singularity
  $\sigma$ at the origin for a smooth vector field $X$
  on $\RR^3$ such that the eigenvalues of $DX(\sigma)$
  are real and satisfy
  $\lambda_1<\lambda_2<0<\lambda_3$. Through a
  coordinate change, we may assume that
  $DX(\sigma)=\diag\{\lambda_1,\lambda_2,\lambda_3\}$. We
  consider the following quadratic forms in $\RR^3$.
  \begin{description}
  \item[Index 1] $\J_1(x,y,z)=-x^2+y^2+z^2$. Then $\J_1$ is
    represented by the matrix $J_1=\diag\{-1,1,1\}$,
    that is, $\J_1(\vec w)=\langle J_1(\vec w),\vec
    w\rangle$ with the canonical inner product.

  Then $\tilde J_1=J_1\cdot
  DX(\sigma)+DX(\sigma)^*\cdot
  J_1=\diag\{-2\lambda_1,2\lambda_2,2\lambda_3\}$ and
  $\tilde J_1'-\delta J_1>0\iff
  2\lambda_1<\delta<2\lambda_2<0$. So $\delta$ must be
  negative and $\tilde J_1$ is not positive
  definite. From Proposition~\ref{pr:J-separated} and
  Theorem~\ref{thm:J-separated-tildeJ} we have
  strict $\J_1$-separation, thus partial hyperbolicity
  with the negative $x$-axis a uniformly contracted
  direction dominated by the $yz$-direction; but this
  is not an hyperbolic splitting.

  Moreover the conclusion would be the same if
  $\lambda_3$ where negative: we get a sink with a
  partially hyperbolic splitting.

  \item[Index 2] $\J_2(x,y,z)=-x^2-y^2+z^2$ represented
    by $J_2=\diag\{-1,-1,1\}$.

  Then $\tilde
  J_2=\diag\{-2\lambda_1,-2\lambda_2,2\lambda_3\}$ and
  $\tilde J_2-\delta J_1>0 \iff
  2\lambda_2<\delta<2\lambda_3$. So $\delta$ might be
  either positive or negative, but we still have strict
  $\J_2$-separation, and $\tilde J_2$ is positive
  definite. Hence by
  Theorem~\ref{thm:J-separated-tildeJ} $X_t$ is
  strictly $\J_2$-monotone at $\sigma$ and the
  splitting
  $\RR^3=(\RR^2\times\{0\})\oplus(\{(0,0)\}\times\RR)$
  is hyperbolic.

  \item[Index 1, not separated]
    $\J_3(x,y,z)=x^2-y^2+z^2$ represented by
    $J_2=\diag\{1,-1,1\}$.

  Now $\tilde
  J_3=\diag\{2\lambda_1,-2\lambda_2,2\lambda_3\}$ is
  not positive definite and $\tilde J_3-\delta J_3$ is
  the diagonal matrix
  $\diag\{2\lambda_1-\delta,-2\lambda_2+\delta,2\lambda_3-\delta\}$
  which represents a positive semidefinite quadratic
  form if, and only if, $\delta\le2\lambda_1,
  \delta\ge2\lambda_2$ and $\delta\le2\lambda_3$, which
  is impossible. Hence we do not have domination of the
  $y$-axis by the $xz$-direction.
\end{description}
\end{example}

\subsection{Organization of the text}
\label{sec:organiz-text}

We study $\J$-separated linear multiplicative cocycles over
flows in Section~\ref{sec:propert-j-separat}, where we prove
the main results whose specialization for the derivative
cocycle of a smooth flow provide the main theorems,
including Proposition~\ref{pr:J-separated}. We then consider
the case of the derivative cocycle and prove
Theorem~\ref{mthm:Jseparated-parthyp} and
Corollaries~\ref{mcor:J-separation-hyperbolicity}
and~\ref{mcor:incompress-anosov} in
Section~\ref{sec:strict-j-separat}. We turn to study
sectional hyperbolic attracting sets in
Section~\ref{sec:approach-via-linear}, where we prove
Theorem~\ref{mthm:2-sec-exp-J-monot} and
Proposition~\ref{pr:P-J-monotonous}.

\subsection*{Acknowledgments}

This is part of the second author PhD Thesis
\cite{luciana-tese} at Instituto de Mate\-m\'atica da
Universidade Federal do Rio de Janeiro (UFRJ), whose
research facilities were most useful. This work have been 
improved during her postdoctoral research at IMPA with financial 
support of INCTMat-CAPES.

%%%%%%%%%%%%%%%%%%%%%%%%%%%%%%%%%%%%%%%%%%%%%%%%%%%%%%%%%%%%

\section{Some properties of quadratic forms and
  $\J$-separated cocycles}
\label{sec:propert-quadrat-forms}

The assumption that $M$ is a compact manifold enables us to
globally define an inner product in $E$ with respect to
which we can find the an orthonormal basis associated to
$\J_x$ for each $x$, as follows. Fixing an orthonormal basis
on $E_x$ we can define the linear operator
\begin{align*}
  J_x:E_x\to E_x \quad\text{such that}\quad \J_x(v)=<J_x
  v,v> \quad \text{for all}\quad v\in T_xM,
\end{align*}
where $<,>=<,>_x$ is the inner product at $E_x$. Since we
can always replace $J_x$ by $(J_x+J_x^*)/2$ without changing
the last identity, where $J_x^*$ is the adjoint of $J_x$
with respect to $<,>$, we can assume that $J_x$ is
self-adjoint without loss of generality.  Hence, we
represent $\J(v)$ by a non-degenerate symmetric bilinear
form $<J_x v,v>_x$.

\subsection{Adapted coordinates for the quadratic form}
\label{sec:adapted-coordin-quad}

Now we use Lagrange's method to
diagonalize this bilinear form, obtaining a base
$\{u_1,\dots,u_n\}$ of $E_x$ such that
\begin{align*}
  \J_x(\sum_{i}\alpha_iu_i)=\sum_{i=1}^q -\lambda_i\alpha_i^2 +
  \sum_{j=q+1}^n \lambda_j\alpha_j^2, \quad
  (\alpha_1,\dots,\alpha_n)\in\RR^n.
\end{align*}
Replacing each element of
this base according to $v_i=|\lambda_i|^{1/2}u_i$ we deduce that
\begin{align*}
\J_x(\sum_{i}\alpha_iv_i)=\sum_{i=1}^q -\alpha_i^2 +
  \sum_{j=q+1}^n \alpha_j^2, \quad
  (\alpha_1,\dots,\alpha_n)\in\RR^n.
\end{align*}
Finally, we can redefine $<,>$ so that the base
$\{v_1,\dots, v_n\}$ is orthonormal. This can be done
smoothly in a neighborhood of $x$ in $M$ since we are
assuming that the quadratic forms are non-degenerate; the
reader can check the method of Lagrange in a standard Linear
Algebra textbook and observe that the steps can be performed
robustly and smoothly for all nearby tangent spaces; see for
instance \cite{ShafRemiz,Postnikov2} and
Example~\ref{ex:rellich} below.

%with small perturbations, for instance in \cite{Maltsev63}.

In this adapted inner product we have that $J_x$ has entries
from $\{-1,0,1\}$ only, $J_x^*=J_x$ and also that
$J_x^2=J_x$.
Having fixed the orthonormal frame as above, the
\emph{standard negative subspace} at $x$ is the one spanned
by $v_{1},\dots, v_{q}$ and the \emph{standard positive
  subspace} at $x$ is the one spanned $v_{q+1},\dots,v_n$.

\begin{example}\label{ex:rellich}
  The example given by Rellich in
  \cite[pp. 52-53]{Rellich69} and adapted by Kato in
  \cite[Example 5.3, Section II.5.4]{Kato95} show that
  diagonalization of symmetric matrices cannot in general be
  done smoothly: consider the following family in
  $\RR^{2\times 2}$
  \begin{align*}
    A_t=
    \begin{pmatrix}
      1-e^{-t^{-2}}\cos(2/t) & -e^{-t^{-2}}\sin(2/t)
      \\
       -e^{-t^{-2}}\sin(2/t) & 1-e^{-t^{-2}}\cos(2/t)
    \end{pmatrix}
    \text{ for $t\neq0$ and  }
    A_0=Id.
  \end{align*}
  The eigenvectors and respective eigenvalues can explicitly
  be calculated for $t\neq0$
  \begin{align*}
    u_t^1&=(\cos(1/t),\sin(1/t))
    \quad\text{belongs to}\quad
    &\lambda_t^1=1-e^{-t^{-2}};
    \\
    u_t^2&=(\sin(1/t),-\cos(1/t))
    \quad\text{belongs to}\quad
    &\lambda_t^2=1+e^{-t^{-2}}.
  \end{align*}
Clearly the orthonormal basis $(u_t^1,u_t^2)$ has no limit
when $t\to0$. 

However, if we consider the quadratic form $Q(X)=X^T A_t X$
with $X=\binom{x}{y}$ we obtain
\begin{align*}
  Q\binom{x}{y}=b_t[x^2+y^2-2a_tyx]=b_t[x^2-2(a_t y)x+(a_ty)^2+y^2-(a_ty)^2]
\end{align*}
where $b_t=1-e^{-t^{-2}}\cos(2/t)$ and
$a_t=\frac{e^{-t^{-2}}\sin(2/t)}{b_t}$ are $C^\infty$ real
maps with a $C^\infty$ extension to $t=0$. Completing the
square in $Q$ we arrive at
  \begin{align}\label{eq:Qsum}
    Q\binom{x}{y}=b_t[(x-a_yy)^2+(1-a_t^2)y^2]=b_tu^2+b_t(1-a_t^2)v^2
  \end{align}
after setting
\begin{align*}
  \binom{u}{v}=
  \begin{pmatrix}
    1 & -a_t \\ 0 & 1
  \end{pmatrix}
  \binom{x}{y}
  \quad\text{or}\quad
  \binom{x}{y}=
  \begin{pmatrix}
    1 & a_t \\ 0 & 1
  \end{pmatrix}
  \binom{u}{v}.
\end{align*}
That is, in the basis $\{w^1_t=(1,0), w^2_t=(a_t,1)\}$ the
quadratic form $Q$ is reduced to a linear combination of
squares (\ref{eq:Qsum}). The choice of basis and the
coefficients of the reduction are not unique (above we might
have completed the square with $y^2$ instead of $x^2$) but
can be chosen smoothly.

Hence the problem of smooth reduction to a sum of squares of
a family of quadratic forms is of a different nature from
the problem of smooth diagonalization of a family of
symmetric linear operators.
\end{example}

\subsubsection{$\J$-symmetrical matrixes and $\J$-selfadjoint operators}
\label{sec:j-symmetr-matrix}

The symmetrical bilinear form defined by $(v,w)=\langle J_x
v,w\rangle$, $v,w\in E_x$ for $x\in M$ endows
$E_x$ with a pseudo-Euclidean structure. Since $\J_x$ is
non-degenerate, then the form $(\cdot,\cdot)$ is likewise
non-degenerate and many properties of inner products are
shared with symmetrical non-degenerate bilinear forms. We
state some of them below.

\begin{proposition}
  \label{pr:propbilinear}
  Let $(\cdot,\cdot):V\times V \to\RR$ be a real symmetric
  non-degenerate bilinear form on the real finite
  dimensional vector space $V$.
  \begin{enumerate}
  \item $E$ is a subspace of $V$ for which $(\cdot,\cdot)$ is
    non-degenerate if, and only if, $V=E\oplus E^\perp$.

    We recall that $E^\perp:=\{v\in V: (v,w)=0
    \quad\text{for all}\quad w\in E\}$, the
    pseudo-orthogonal space of $E$, is defined using the
    bilinear form.
  \item Every base $\{v_1,\dots,v_n\}$ of $V$ can be
    orthogonalized by the usual Gram-Schmidt process of
    Euclidean spaces, that is, there are linear combinations
    of the basis vectors $\{w_1,\dots, w_n\}$ such that they
    form a basis of $V$ and
    $(w_i,w_j)=0$ for $i\neq j$.  Then this last base can be
    pseudo-normalized: letting $u_i=|(w_i,w_i)|^{-1/2}w_i$ we
    get $(u_i,u_j)=\pm\delta_{ij}, i,j=1,\dots,n$.
  \item There exists a maximal dimension $p$ for a subspace
    $P_+$ of $\J$-positive vectors and a maximal dimension
    $q$ for a subspace $P_-$ of $\J$-negative vectors;
    we have $p+q=\dim V$ and $q$ is known
    as the \emph{index} of $\J$.
  \item For every linear map $L:V\to\RR$ there exists a
    unique $v\in V$ such that $L(w)=(v,w)$ for each $w\in V$.
  \item For each $L:V\to V$ linear there exists a unique
    linear operator $L^+:V\to V$ (the pseudo-adjoint) such that
    $(L(v),w)=(v,L^+(w))$ for every $v,w\in V$.
  \item Every pseudo-self-adjoint $L:V\to V$, that is,
    such that $L=L^+$, satisfies
    \begin{enumerate}
    \item eigenspaces corresponding to distinct eigenvalues
      are pseudo-orthogonal;
    \item if a subspace $E$ is $L$-invariant, then $E^\perp$
      is also $L$-invariant.
    \end{enumerate}
  \end{enumerate}
\end{proposition}

The proofs are rather standard and can be found in
\cite{Maltsev63}.

\subsection{Properties of $\J$-separated cocycles}
\label{sec:propert-j-separat}

In what follows we usually drop the subscript indicating the
point where $\J$ is calculated to avoid heavy notation,
since the base point is clear from the context.

\subsubsection{$\J$-separated linear maps}
\label{sec:j-separat-linear}

The following simple result will be very useful in what follows.

\begin{lemma}
  \label{le:kuhne}
  Let $V$ be a real finite dimensional vector space endowed
  with a indefinite non-degenerate quadratic form
  $\J:V\to\RR$.

If a symmetric bilinear form $F:V\times V\to\RR$ is
non-negative on $C_0$ then
\begin{align*}
  r_+=\inf_{v\in C_+} \frac{F(v,v)}{\langle Jv,v\rangle}
  \ge \sup_{u\in C_-}\frac{F(u,u)}{\langle Ju,u\rangle}=r_-
\end{align*}
and for every $r$ in $[r_-,r_+]$ we have
$F(v,v)\ge r\langle Jv,v\rangle$ for each vector $v$.

In addition, if $F(\cdot,\cdot)$ is positive on
$C_0\setminus\{\vec0\}$, then $r_-<r_+$ and $F(v,v)>
r\langle Jv,v\rangle$ for all vectors $v$ and $r\in(r_-,r_+)$.
\end{lemma}

\begin{proof}
This can be found in ~\cite{Wojtk01} and also
in~\cite{Pota79}. We present the simple proof here for
completeness.

Let us assume that the $F$ is non-negative on $C_0$ and
argue by contradiction: we also assume that
\begin{align}\label{eq:opposite}
  \inf_{v\in C_+} \frac{F(v,v)}{\langle Jv,v\rangle}
  < \sup_{u\in C_-}\frac{F(u,u)}{\langle Ju,u\rangle}.
\end{align}
Hence we can find $v_0\in C+$ and $u_0\in C_-$ with $\J(v_0)=1$
and $\J(u_0)=-1$ such that $F(v_0,v_0)+F(u_0,u_0)<0$. We can also
find an angle $\alpha$ such that both linear combinations
\begin{align*}
  v=v_0\cos\alpha+u_0\sin\alpha \qand
  w=-v_0\sin\alpha+u_0\cos\alpha
\end{align*}
belong to $C_0$. Then we must have $F(v,v)\ge0$ and
$F(w,w)\ge0$, but we also have
\begin{align*}
  F(v,v)+F(w,w)
  &=
  \cos^2\alpha\cdot F(v_0,v_0)+
  +\sin2\alpha\cdot F(u_0,v_0) + \sin^2\alpha \cdot
  F(u_0,u_0)
  \\
  &\quad+
  \sin^2\alpha\cdot F(v_0,v_0)-\sin2\alpha\cdot F(u_0,v_0)+
  \cos^2\alpha\cdot F(u_0,u_0)
  \\
  &=
  F(v_0,v_0)+F(u_0,u_0)<0
\end{align*}
and this contradiction shows that the opposite of
(\ref{eq:opposite}) must be true.

Analogously, if $F$ is positive on $C_0\setminus\{\vec0\}$,
then we can argue in the same way: we assume that
(\ref{eq:opposite}) is true with $\le$ in the place of $<$;
we obtain $F(v_0,v_0)+F(u_0,u_0)\le0$ and then construct
$v,w$ such that $F(v,v)+F(w,w)>0$; and, finally, we show
that $F(v,v)+F(w,w)= F(v_0,v_0)+F(u_0,u_0)\le0$ to arrive
again at a contradiction.
\end{proof}

\begin{remark}
  \label{rmk:Jseparated}
  Lemma~\ref{le:kuhne} shows that if $F(v,w)=\langle \tilde J
  v,w\rangle$ for some self-adjoint operator $\tilde J$ and
  $F(v,v)\ge0$ for all $v$ such that $\langle J v,
  v\rangle=0$, then we can find $a\in\RR$ such that
  $\tilde J \ge a J$. This means precisely that $\langle
  \tilde J v,v\rangle\ge a\langle Jv, v\rangle$ for all
  $v$.

  If, in addition, we have $F(v,v)>0$ for all $v$ such that
  $\langle J v, v\rangle=0$, then we obtain a strict
  inequality $\tilde J > a J$ for some $a\in\RR$ since the
  infimum in the statement of Lemma~\ref{le:kuhne} is
  strictly bigger than the supremum.
\end{remark}

The (longer) proofs of the following results can be found
in~\cite{Wojtk01} or in~\cite{Pota79}; see also~\cite{Wojtk09}.

\begin{proposition}
  \label{pr:J-separated-spectrum}
  Let $L:V\to V$ be a $\J$-separated linear operator. Then
  \begin{enumerate}
  \item $L$ can be uniquely represented by $L=RU$, where $U$
    is a $\J$-isometry (i.e. $\J(U(v))=\J(v), v\in V$) and
    $R$ is $\J$-symmetric (or $\J$-pseudo-adjoint; see
    Proposition~\ref{pr:propbilinear}) with positive
    spectrum.
  \item the operator $R$ can be diagonalized by a
    $\J$-isometry. Moreover the eigenvalues of $R$ satisfy
    \begin{align*}
      0<r_-^q\le\dots\le r_-^1=r_-\le r_+=r_1^+\le\dots\le r_+^p.
    \end{align*}
  \item the operator $L$ is (strictly) $\J$-monotonous if,
    and only if, $r_-\le (<) 1$ and $r_+\ge (>) 1$.
  \end{enumerate}
\end{proposition}

For a $\J$-separated operator $L:V\to V$ and a
$d$-dimensional subspace $F_+\subset C_+$, the subspaces
$F_+$ and $L(F_+)\subset C_+$ have an inner product given by
$\J$. Thus both subspaces are endowed with volume
elements. Let $\alpha_d(L;F_+)$ be the rate of expansion of
volume of $L\mid_{F_+}$ and $\sigma_d(L)$ be the infimum of
$\alpha_d(L;F_+)$ over all $d$-dimensional subspaces $F_+$
of $C_+$.

\begin{proposition}
  \label{pr:product-vol-exp}
  We have $\sigma_d(L)=r_+^1 \cdots r_+^d$, where $r^i_+$
  are given by Proposition~\ref{pr:J-separated-spectrum}(2).

  Moreover, if $L_1,L_2$ are $\J$-separated, then
  $\sigma_d(L_1L_2)\ge\sigma_d(L_1)\sigma_d(L_2)$.
\end{proposition}

The following corollary is very useful.

\begin{corollary}
  \label{cor:compos-max-exp}
  For $\J$-separated operators $L_1,L_2:V\to V$ we have
  \begin{align*}
    r_+^1(L_1L_2)\ge r_+^1(L_1) r_+^1(L_2) \quad\text{and}\quad
    r_-^1(L_1L_2)\le r_-^1(L_1)r_-^1(L_2).
  \end{align*}
  Moreover, if the operators are strictly $\J$-separated,
  then the inequalities are strict.
\end{corollary}

%%%%%%%%%%%%%%%%%%%%%%%%%%%%%%%%%%%%%%%%%%%%%%%%%%%%

\subsection{$\J$-separated linear cocycles over flows}
\label{sec:j-separat-linear-1}

The results in the previous subsection provide the following
characterization of $\J$-separated cocycles $A_t(x)$ over a
flow $X_t$ in terms of the infinitesimal generator $D(x)$ of
$A_t(x)$; see~\eqref{eq:infinitesimal-gen}. The following
statement is more precise than
Proposition~\ref{pr:J-separated}.

Let $A_t(x)$ a linear multiplicative cocycles over a flow $X_t$. We define
the infinitesimal generator of $A_t(x)$ by
\begin{align}\label{eq:infinitesimal-gen}
  D(x):=\lim_{t\to0}\frac{A_t(x)-Id}t.
\end{align}

\begin{theorem}
  \label{thm:J-separated-tildeJ}
  Let $X_t$ be a flow defined on a positive invariant subset
  $U$, $A_t(x)$ a cocycle over $X_t$ on $U$ and $D(x)$ its
  infinitesimal generator. Then
  \begin{enumerate}
  \item $\partial_t \J(A_t(x)v) = \langle \tilde J_{X_t(x)}
    A_t(x)v,A_t(x)v\rangle$ for all $v\in E_x$ and $x\in U$,
    where
    \begin{align}\label{eq:J-separated-tildeJ}
      \tilde J_x:= J\cdot D(x) + D(x)^* \cdot J
    \end{align}
    and $D(x)^*$ denotes the adjoint of the linear map
    $D(x):E_x\to E_x$ with respect to the adapted inner
    product at $x$;
  \item the cocycle $A_t(x)$ is $\J$-separated if, and only
    if, there exists a neighborhood $V$ of $\Lambda$,
$V\subset U$ and a function $\delta:V\to\RR$ such that
    \begin{align}\label{eq:J-ge}
      \tilde J_x\ge\delta(x) J
      \quad\text{for all}\quad x\in V.
    \end{align}
    In particular we get 
  $\partial_t\log|\J(A_t(x)v)|\ge\delta(X_t(x)), x\in V,
  t\ge0, v\in E_x$ and $\J(v)>0$; or
  $\partial_t\log|\J(A_t(x)v)|\le\delta(X_t(x)), x\in V,
  t\ge0, v\in E_x$ and $\J(v)<0$;
% $\partial_t\log|\J(A_t(x)v)|\ge
%     \delta(X_t(x))$, $v\in E_x, x\in V, t\ge0$;
  \item if the inequalities in the previous item are strict,
    then the cocycle $A_t(x)$ is strictly
    $\J$-separated. Reciprocally, if $A_t(x)$ is strictly
    $\J$-separated, then there exists compatible field
    of forms $\J_0$ on $V$ satisfying the strict
    inequalities of item (2).
  \item Define the function
    \begin{align}\label{eq:delta-area}
     \Delta_s^t(x):=\int_s^t\delta(X_s(x))\,ds.
    \end{align}
    For a $\J$-separated cocycle $A_t(x)$, we have
    $\frac{|\J(A_{t_2}(x)v)|}{|\J(A_{t_1}(x)v)|}\ge \exp
    \Delta_{t_1}^{t_2}(x)$ for all $v\in E_x$ and reals
    $t_1<t_2$ so that $\J(A_t(x)v)\neq0$ for all $t_1\le
    t\le t_2$.
  % \item if $A_t(x)$ is $\J$-separated and $x\in\Lambda(U), v\in
  %   C_+(x)$ and $w\in C_-(x)$ are non-zero vectors, then for
  %   every $t>0$ such that $A_s(x)w\in C_-(X_s(x))$ for all
  %   $0<s<t$
  % \begin{align}\label{eq:Jquotient-upper}
  %   \frac{|\J(A_t(x)w)|}{\J(A_t(x)v)}
  %   \le
  %   \frac{|\J(w)|}{\J(v)} \exp \big(2\Delta_0^t(x)).
  % \end{align}
\item  we can bound $\delta$ at every $x\in\Gamma$
  by $ \sup_{v\in C_-(x)}\frac{\J'(v)}{\J(v)}
  \le\delta(x)\le \inf_{v\in
    C_+(x)}\frac{\J'(v)}{\J(v)}.$
  \end{enumerate}
\end{theorem}

\begin{remark}
  \label{rmk:semipositive}
  If $\delta(x)=0$, then $\tilde \J_x$ is positive
  semidefinite operator. But for $\delta(x)\neq0$ the symmetric operator
  $\tilde \J_x$ might be an indefinite quadratic form.
\end{remark}

\begin{remark}
  \label{rmk:reciprocal-item3}
  The necessary condition in item (3) of
  Theorem~\ref{thm:J-separated-tildeJ} is proved in
  Section~\ref{sec:partial-hyperb-impli} after
  Theorem~\ref{thm:necessity} and
  Proposition~\ref{pr:strict-J-sep}.
\end{remark}

\begin{remark}
  \label{rmk:axcont}
  % Considering
  % \begin{align*}
  %   a:= \inf\{\langle \tilde J_x v,v\rangle/\langle Jv,
  %   v \rangle: {v\in C_+}\} \ \text{and} \ b:=
  %   \sup\{\langle \tilde J_x v,v\rangle/\langle Jv, v
  %   \rangle: {v\in C_-}\},
  % \end{align*}
  % the arguments above show that $\delta(x)$ can be
  % taken in $\left[ a , b \right]$ in the $\J$-separated
  % case; and in the interior of the above interval in
  % the strictly $\J$-separated case.
  We can take $\delta(x)$ as a continuous function of
  the point $x\in U$ by the last item of
  Theorem~\ref{thm:J-separated-tildeJ}.
\end{remark}

\begin{remark}
  \label{rmk:strictly-J-separated}
  Complementing Remark~\ref{rmk:J-separated-C-}, the
  necessary and sufficient condition in items (2-3) of
  Theorem~\ref{thm:J-separated-tildeJ}, for (strict)
  $\J$-separation, shows that a cocycle $A_t(x)$ is
  (strictly) $\J$-separated if, and only if, its inverse
  $A_{-t}(x)$ is (strictly) $(-\J)$-separated.
\end{remark}

\begin{remark}
  \label{rmk:Jexp-ineq}
  Item (5)  of
  Theorem~\ref{thm:J-separated-tildeJ}
  shows that $\delta$ is a measure of the ``minimal
  instantaneous expansion rate'' of $|\J\circ A_t(x)|$
  on positive vectors, but also
  a ``maximal instantaneous expansion
  rate'' of $|\J\circ A_t(x)|$ on negative vectors.%  and
  % the last inequality shows in addition that $\delta$
  % is also a bound for the ``instantaneous variation of
  % the ratio'' between $|\J\circ A_t(x)|$ on negative
  % and positive vectors.
  % For the maximal
  % expansion rate on positive vectors we can use that $\tilde
  % J_x$ is cleary a self-adjoint operator with respect to the
  % standard inner product $\langle\cdot,\cdot\rangle$.

  Hence, the behavior of the area under the function
  $\delta$, given by
  $\Delta_s^t(x)$ as $t-s$ tends
  to $\pm\infty$, defines the type of partial
  hyperbolic splitting (with contracting or expanding
  subbundles) exhibited by a strictly $\J$-separated
  cocycle.

  In this way we have conditions ensuring partial
  hyperbolicity of an invariant subset involving only the
  spatial derivative map of the vector field; see Theorem
  \ref{thm:char-dom-split} in
  Section~\ref{sec:charact-splitt-throu}.
\end{remark}

\begin{proof}[Proof of Theorem~\ref{thm:J-separated-tildeJ}]
The map $\psi(t,v):=\langle J A_t(x) v,A_t(x) v\rangle$ is
smooth and for $v\in E_x$ satisfies
 \begin{align*}
  \partial_t\psi(t,v) &= \langle \big(J\cdot
  D(X_t(x))\big)A_t(x) v, A_t(x) v \rangle
  \\
  &\quad + \langle J\cdot A_t(x) v , D(X_t(x)) A_t(x) v \rangle
  \\
  &= \langle \big(J\cdot D(X_t(x)) + D(X_t(x))^* \cdot J
  \big)A_t(x) v, A_t(x) v \rangle,
  % &= \langle \big(J\cdot D(x) + D(x)^* \cdot J \big) v, v
  % \rangle,\nonumber
\end{align*}
where we have used the fact that the cocycle has an
infinitesimal generator $D(x)$: we have the relation
\begin{align}\label{eq:linear-variational}
  \partial_t A_t(x) v = D(X_t(x))\cdot A_t(x) v
  \quad\text{for all
  $t\in\RR, x\in M$ and $v\in E_x$.}
\end{align}
This is because we have
the linear variation equation: $A_t(x)$ is the solution of
the following non-autonomous linear equation
\begin{align}\label{eq:non-homo-ode}
  \begin{cases}
    \dot{Y}=D(X_t(x))Y\\
    Y(0)=Id
  \end{cases}.
\end{align}
We note that the argument does not change for $x=\sigma$ a singularity of $X_t$.

This proves the first item of the statement of the theorem.

We observe that the independence of $J$ from $X_t(x)$ is a
consequence of the choice of adapted coordinates and inner
product, since in this setting the operator $J$ is
fixed. However, in general, this demands the rewriting of
the cocycle in the coordinate system adapted to $\J$.

For the second item, let us assume that $A_t(x)$ is
$\J$-separated on $U$. Then, by definition, if we fix $x\in
U$
\begin{align}\label{eq:def-J-separated}
  \langle J A_t(x) v,A_t(x) v\rangle > 0 \quad \text{for all
    $t>0$ and all $v\in E_x$ such that $\langle
    Jv,v\rangle>0$}.
\end{align}
We also note that, by continuity, we have $\langle J A_t(x)
v,A_t(x) v\rangle \ge 0$ for all $v$ such that $\langle
Jv,v\rangle=0$. Indeed, for any given $t>0$ and $v\in C_0$
we can find $w\in C_+$ such that $v+w\in C_+$. Then we have
$\langle J A_t(x) (v+\lambda w),A_t(x) (v+\lambda w)\rangle
> 0$ for all $\lambda>0$, which proves the claim letting
$\lambda$ tend to $0$.

The map $\psi(t,v)$ satisfies $\psi(0,v)=0\le\psi(t,v)$ for
all $t>0$ and $v\in C_0(x)$, hence from the first item
already proved
 \begin{align*}
  0 \le
  \partial_t\psi(t,v)\mid_{t=0}
  = \langle \big(J\cdot D(x) + D(x)^* \cdot J \big) v, v
  \rangle.
\end{align*}

According to Lemma~\ref{le:kuhne} (cf. also
Remark~\ref{rmk:Jseparated}) there exists $\delta(x)\in\RR$
such that \eqref{eq:J-ge} is true and this, in turn, implies
that $\partial_t\J(A_t(x)v)\ge \delta(x)\J(A_t(x)v)$, for
all $v\in E_x, x\in U, t\ge0$. This completes the proof of
necessity in the second item.

Moreover, from Lemma~\ref{le:kuhne} we
have that $\delta(x)$ satisfies the inequalities in
item (5).

To see that this is a sufficient condition for
$\J$-separation, let $\tilde \J_x\ge \delta(x)\J$, for some function
$\delta:U\to\RR$. Then, for all $v\in E_x$ such that $\langle J
v, v\rangle>0$, since $\partial_t\J(A_t(x)v) \ge \delta(X_t(x))
\J(A_t(x)v)$, we obtain
\begin{align}\label{eq:Jexponential-inequality}
  |\J(A_t(x)v)|\ge |\J(v)|\exp\left( \int_0^t
    \delta(X_s(x))\,ds\right) = |\J(v)|\exp \Delta(x,t) >0 , \quad t\ge0,
\end{align}
and $\J(A_t(x)v)>0$ for all $t>0$ by continuity. This shows
that $A_t(x)$ is $\J$-separated. This completes the proof of
the second item in the statement of the theorem.

For the third item, we only prove the first statement and
leave the longer proof of the second statement for
Section~\ref{sec:partial-hyperb-impli} in
Proposition~\ref{pr:strict-J-sep}.  If $\tilde J_x>
\delta(x)J$ for all $x\in U$, then for $t>0$ we
obtain~(\ref{eq:Jexponential-inequality}) with strict
inequalities for $v\in C_0(x)$, hence $\J(A_t(x)v)>0$ for
$t>0$.  So $A_t(x)$ is strictly $\J$-separated.

For the fourth item, we just itegrate the inequality of item
(2) from $s$ to $t$ in the real line. The proof is complete.
%  subitem (a) is just item (2) with the
% observation that a positive vector remains positive under
% the action of the cocycle for positive $t$, thus we get
% \eqref{eq:Jexponential-inequality} without the modulus
% signs. As for subitem (b), for each $x_0\in U$ and $t_0>0$
% such that $X_{-t}(x_0)\in U$ for all $0\le t\le t_0$, fixing
% $x=X_{-t}(x_0)$ and $w=A_{-t}(x_0)w_0$ for some $w_0\in
% C_-(x_0)$ and $t\in[0,t_0]$, we get from items (1) and (2)
% already proved
% \begin{align*}
%   \partial_t\J(A_t(x)w)&\ge\delta(X_t(x))\J(A_t(x)w)
%   \qand
%   \J(A_t(x)w)<0,
%   \\\text{so}\quad
%   \frac{\partial_t\J(A_t(x)w)}{\J(A_t(x)w)}
%   &\le\delta(X_t(x)).
% \end{align*}
% Integrating this inequality from $0$ to $t$ we obtain
% \begin{align}\label{eq:quotupper}
%   \log\left|\frac{\J(A_t(x)w)}{\J(w)} \right| \le
%   \int_0^t\delta(X^s(x))\,ds
%   \quad\text{thus}\quad
%   |\J(A_t(x)w)|
%     \le
%     |\J(w)|\exp \Delta(x,t).
% \end{align}

% For the fifth item, we calculate the derivative of the ratio
% of the forms and use the previous results as follows
% \begin{align*}
%   \partial_t\left(\frac{|\J(A_t(x)w)|}{\J(A_t(x)v)}\right)
%   &=
%   -\frac{\langle \tilde J_{X_t(x)}A_t(x)w,A_t(x)w \rangle}{\J(A_t(x)v)}
%   -\frac{\J(A_t(x)w)}{\J(A_t(x)v)}\cdot
%   \frac{\langle \tilde J_{X_t(x)}A_t(x)v,A_t(x)v
%     \rangle}{\J(A_t(x)v)}
%   \\
%   &\le
%   -\delta(X_t(x))\frac{\J(A_t(x)w)}{\J(A_t(x)v)}
%   -\frac{\J(A_t(x)w)}{\J(A_t(x)v)}\cdot \delta(X_t(x))
%   \\
%   &=
%   2\delta(X_t(x))\cdot\frac{|\J(A_t(x)w)|}{\J(A_t(x)v)}
% \end{align*}
% and the result is obtained by integrating this equation from $0$
% to $t$.
\end{proof}

In Section~\ref{sec:charact-splitt-throu} we show that the
asymptotic behavior of the function $\Delta_s^t(x)$
as $t-s$ grows to $\pm\infty$ defines the type of partial
hyperbolic splitting exhibited by a strictly $\J$-separated
cocycle.

In this way we have a condition ensuring partial and uniform hyperbolicity of
an invariant subset involving only the vector field and its
spatial derivative map.

% It is easy to see that we can replace $D$ by its adjoint in
% the inequality~\eqref{eq:J-separated-tildeJ}.

% \begin{corollary}
%   \label{cor:J-separated-DXstar}
%   In the same setting of
%   Theorem~\ref{thm:J-separated-tildeJ},
%   inequality~\eqref{eq:J-separated-tildeJ} is equivalent to
%   \begin{align}
%     \label{eq:inequality-tildeJstar}
%     J\cdot D(x)^* + D(x) \cdot J \ge \delta(x) J
%   \end{align}
%   with the same function $\delta:U\to[0,+\infty)$. In particular,
%   this means that $\big(\tilde \J_{X_t(x)}\big)^*\ge \delta(X_t(x)) \J$ for
%   all $t>0$ and $x\in U$.
% \end{corollary}

% \begin{proof}
%   We use that $J^2=J=J^*$ to write, for any tangent vector $v$
%   \begin{align*}
%     \delta(x)\langle Jv ,v \rangle
%     =
%     \delta(x)\langle J^3 v, v \rangle
%     =
%     \delta(x)\langle J (Jv), Jv\rangle
%     \le
%     \langle \tilde J_x (Jv) ,(Jv) \rangle
%     =
%     \langle \tilde J_x J v , Jv \rangle
%     =
%     \langle J\tilde J_x J v,v\rangle
%   \end{align*}
%   and $J\tilde J_x J = J\cdot DX(x)^* + DX(x) \cdot J$.
% \end{proof}

\subsection{Strict $\J$-separated cocycles and domination}
\label{sec:j-separat-cocycl}

We assume from now on that a family $A_t(x)$ of linear
multiplicative cocycles on a vector bundle $E_U$ over the
flow $X_t$ on a trapping region $U\subset M$ has been given,
together with a field of non-degenerate quadratic forms $\J$ 
on $E_U$ with constant index $q<\dim E_U$.

\begin{theorem}
  \label{thm:strict-J-separated-cocycle}
  The cocyle $A_t(x)$ is strictly
  $\J$-separated if, and only if, $E_U$ admits a
  %n invariant and continuous
  dominated splitting $F_-\oplus F_+$ with
  respect to $A_t(x)$ on the maximal invariant subset
  $\Lambda$ of $U$, with constant
  dimensions $\dim F_-=q, \dim F_+=p, \dim M = p+q $.
\end{theorem}

Moreover the properties stated in
Theorem~\ref{thm:strict-J-separated-cocycle} are robust:
they hold for all nearby cocycles on $E_U$ over all flows
close enough to $X_t$; see
Section~\ref{sec:partial-hyperb-impli}.

We now start the proof of
Theorem~\ref{thm:strict-J-separated-cocycle}. We construct a
decomposition of the tangent space over $\Lambda$ into a
direct sum of invariant subspaces and then we prove that
this is a dominated splitting.

\subsubsection{The cones are contracted}
\label{sec:project-metrics}

To obtain the invariant subspaces, we show that the action
of $A_{t}(x)$ on the set of all $p$-dimensional spaces
inside the positive cones is a contraction in the
appropriate distance. For that we use a result from
\cite{Wojtk01}.

Let us fix $C_+=C_+(x)$ for some $x\in\Lambda$ and consider
the set $G_p(C_+)$ of all $p$-subspaces of $C_+$, where
$p=n-q$. This manifold can be identified with the set of
all $q\times p$ matrices $T$ with real entries such that
$T^*T<I_p$, where $I_p$ is the $p\times p$ identity matrix
and $<$ indicates that for the standard inner product in
$\RR^p$ we have $<T^*T u,u> \ < \ <u,u>$, for all $u\in\RR^p$.

A $\J$-separated operator naturally sends $G_p(C_+)$ inside
itself. This operation is a contraction.

\begin{theorem}
  \label{thm:proj-contr}
  There exists a distance $\dist$ on $G_p(C_+)$ so that
  $G_p(C_+)$ becomes a complete metric space and,
  if $L:V\to V$ is $\J$-separated and $T_1,T_2\in G_p(C_+)$,
  then
  $$
  \dist(L(T_1),L(T_2))\le \frac{r_-}{r_+}\dist(T_1,T_2),
  $$
  where $r_\pm$ are given by Proposition~\ref{pr:J-separated-spectrum}.
\end{theorem}

\begin{proof}
  See \cite[Theorem 1.6]{Wojtk01}.
\end{proof}

\subsubsection{Invariant directions}
\label{sec:invari-direct}

Now we consider a pair $C_-(x)$ and $C_-(X_{-t}(x))$ of positive
cones, for some fixed $t>0$ and $x\in\Lambda$, together with
the linear isomorphism $A_{-t}(x): E_x\to E_{X_{-t}(x)}$. We
note that the assumption of strict $\J$-separation ensures
that $A_{-t}(x)\mid C_-(x): C_-(x)\to
C_-(X_{-t}(x))$. We have in fact
\begin{align}\label{eq:strict-inclusion}
  \overline{A_{-t}(x)\cdot C_-(x)}\subset C_-(X_{-t}(x)).
\end{align}
Moreover, by Theorem~\ref{thm:proj-contr} we have that the
diameter of $A_{-nt}(x)\cdot C_-(X_{nt}(x))$ decreases
exponentially fast when $n$ grows. Hence there exists a
unique element $F_-(x)\in G_q(C_-(x))$ in the intersection
of all these cones. Analogous results hold for the positive
cone with respect to the action of $A_t(x)$. It is easy to
see that
\begin{align}\label{eq:F-Axt-inv}
  A_{t}(x)\cdot F_\pm(x)=F_\pm(X_{t}(x)), \quad x\in\Lambda.
\end{align}
Moreover, since the strict
inclusion~\eqref{eq:strict-inclusion} holds for whatever
$t>0$ we fix, then we see that the subspaces $F_\pm$ do not
depend on the chosen $t>0$.

\subsubsection{Domination}
\label{sec:dominat}

The contraction property on $C_+$ for $A_t(x)$ and on $C_-$
for $A_{-t}(x)$, any $t>0$, implies domination
directly. Indeed, let us fix $t>0$ in what follows and
consider the norm $|\cdot|$ induced on $E_x$ for each $x\in
U$ by
\begin{align*}
  |v|:=\sqrt{\J(v_-)^2+\J(v_+)^2} \quad\text{where}\quad
  v=v_-+ v_+, v_\pm\in F_\pm(x).
\end{align*}
Now, according to Lemma~\ref{le:kuhne} together with
Proposition~\ref{pr:J-separated-spectrum} we have that, for
each $x\in \overline{X_t(U)}$ and every pair of unit vectors
$u\in F_-(x)$ and $v\in F_+(x)$
\begin{align*}
  \frac{|A_t(x)u|}{|A_{t}(x) v|}\le
  \frac{r_-^t(x)}{r_+^t(x)}\le
  \omega_t:=\sup_{z\in\overline{X_t(U)}}\frac{r_-^t(z)}{r_+^t(z)}   <1,
\end{align*}
where $r_\pm^t(x)$ represent the values $r_\pm$ shown to
exist by Lemma~\ref{le:kuhne} with respect to the strictly
$\J$-separated linear map $A_t(x)$. The value of $\omega_t$ is
strictly smaller than $1$ by continuity of the functions
$r_\pm$ on the compact subset $\overline{X_t(U)}$.

Now we use the following well-known lemma.

\begin{lemma}\label{le:dom-split-equiv}
  Let $X_t$ be a $C^1$ flow and $\Lambda$ a compact
  invariant set for $X_t$ admitting a continuous invariant
  splitting $T_{\Lambda}M = F_- \oplus F_+$. Then this
  splitting is dominated if, and only if, there exists a
  riemannian metric on $\Lambda$ inducing a norm such that
  $$
  \lim_{t\to+\infty}\|A_t(x)\mid_{F_-(x)}\|
  \cdot\|A_{-t}(X_t(x))\mid_{F_+(X_t(x))}\| = 0,
  $$
  for all $x \in \Lambda$.
\end{lemma}

% \begin{proof}
%   See \cite{ArArbSal} but we give the details in
%   Appendix~\ref{sec:equival-condit-domin}.
% \end{proof}

This shows that for the cocycle $A_t(x)$ the splitting
$E_\Lambda=F_-\oplus F_+$ is dominated, since the above argument
does not depend on the choice of $t>0$ and implies that
\begin{align*}
 \lim\limits_{n\to+\infty}\frac{ |A_{nt}(x)u|}{|A_{nt}(x) v|} = 0,
\end{align*}
and we conclude that
\begin{align*}
\lim\limits_{t\to+\infty}  \frac{|A_{t}(x)u|}{|A_{t}(x) v|} = 0,
\quad x\in\Lambda, u\in F_-(x), v\in F_+(x).
\end{align*}

\subsubsection{Continuity of the splitting}
\label{sec:continuity-splitting}

The continuity of the subbundles $F_\pm$ over $\Lambda$ is a
consequence of domination together with the observation that
the dimensions of $F_\pm(x)$ do not depend on $x\in
\Lambda$; see for example \cite[Appendix
B]{BDV2004}. Moreover, since we are assuming that $A_t(x)$
is smooth, i.e. the cocycle admits an infinitesimal
generator, then the $A_t(x)$-invariance ensures that the
subbundles $F_\pm(x)$ can be differentiated along the orbits
of the flow.

This completes the proof that strict $\J$-separation implies
a dominated splitting, as stated in
Theorem~\ref{thm:strict-J-separated-cocycle}.

%%%%%%%%%%%%%%%%%%%%%%%%%%%%%%%%%%%%%%%%%%%%%%%%%%%%%%%%%%%%%%%

\subsection{Domination implies strict $\J$-separation}
\label{sec:partial-hyperb-impli}

Now we start the proof of the converse of
Theorem~\ref{thm:strict-J-separated-cocycle} by showing
that, given a dominated decomposition of a vector bundle
over a compact invariant subset $\Lambda$ of the base, for a
$C^1$ vector field $X$ on a trapping region $U$, there
exists a smooth field of quadratic forms $\J$ for which $Y$
is strictly $\J$-separated on $E_V$ over a neighborhood $V$
of $\Lambda$ for each vector field $Y$ sufficiently $C^1$
close to $X$ and every cocycle close enough to $A_t$.

We define a distance between smooth cocycles as follows. If
$D_A(x), D_B(x):E_x\to E_x$ are the infinitesimal generators
of the cocycles $A_t(x),B_t(x)$ over the flow of $X$ and $Y$
respectively, then we can recover the cocycles through the
non-autonomous ordinary differential
equation~\eqref{eq:non-homo-ode}. We then define the
\emph{distance $d$ between the cocycles $A_t$ and $B_t$} to
be
\begin{align*}
  d\big((A_t)_t,(B_t)_t\big):=\sup_{x\in M}\|D_A(x)-D_B(x)\|,
\end{align*}
where $\|\cdot\|$ is a norm on the vector bundle $E$. We
always assume that we are given a Riemannian inner product
in $E$ which induces the norm $\|\cdot\|$.

As before, let $\Lambda=\Lambda(U)$ be a maximal positively
invariant subset for a $C^1$ vector field $X$ endowed with a
linear multiplicative cocycle $A_t(x)$ defined on a vector
bundle over $U$.It is well-known that attracting sets are persistent in the
following sense.  Let $U$ be a trapping region for the flow
of $X$ and $\Lambda(U)=\Lambda_X(U)$ the corresponding
attracting set.
\begin{lemma}\cite[Chapter 10]{robinson1999}.
  \label{le:trapping}
   There exists a neighborhood $\U$ of $X$ in
  $\Mundo$ and an open neighborhood $V$ of $\Lambda(U)$
  such that $V$ is a trapping region for all $Y\in\U$, that
  is, there exists $t_0>0$ for which
  \begin{itemize}
  \item $Y_t(V)\subset V\subset U$ for all $t>0$;
  \item $\overline{Y_t(V)}\subset V$ for all $t>t_0$; and
  \item $\overline{Y_t(V)}\subset U$ for all $t>0$.
  \end{itemize}
\end{lemma}
We can thus consider
$\Lambda_Y=\Lambda_Y(U)=\cap_{t>0}\overline{Y_t(U)}$ in what follows
for $Y\in\V$ in a small enough $C^1$ neighborhood of $X$.

\begin{theorem}
  \label{thm:necessity}
  Suppose that $\Lambda$ has a dominated splitting
  $E_\Lambda=F_-\oplus F_+$. Then there exists a $C^1$
  field of quadratic forms $\J$ on a neighborhood
  $V\subset U$ of $\Lambda$, a $C^1$-neighborhood $\V$
  of $X$ and a $C^0$-neighborhood $\W$ of $A_t(x)$ such
  that $B_t(x)$ is strictly $\J$-separated on $V$ with
  respect to $Y\in\V$ and $B\in\W$. More
    precisely, there are constants $\kappa,\omega>0$
    such that, for each $Y\in\V$, $B\in\W$,
    $x\in\Lambda_Y$ and $t\ge0$
    \begin{align*}
      |\J(B_t(x)v_-)|\le \kappa e^{-\omega
        t}\J(B_t(x)v_+) ,  v_\pm\in F^B_\pm(x),
      \quad
      \J(v_\pm)=\pm1;
    \end{align*}
    where $F^B_\pm$ are the subbundles of the dominated
    splitting of $E_{\Lambda_Y}$.
\end{theorem}

The quadratic form $\J$ is a inner product in each $F_\pm$,
since $F_\pm$ are finite dimensional subbundles of $E$
where $\J$ does not change sign, thus
the compactness of $\Lambda$ ensures
 the following.

\begin{lemma}\label{le:ratio}
  There exists a constant $K>0$ such that for every
  pair of non-zero vectors $(w,v)\in F_-(x)\times
  F_+(x)$ we have $ \frac1K\|w\|^2 \le |\J(w)| \le
  K\|w\|^2, \frac1K\|v\|^2 \le \J(v) \le K\|v\|^2$ and
    \begin{align*}
    \frac1K
    \sqrt{\frac{|\J(w)|}{\J(v)}}
    &\le
    \frac{\| w\|}{\| v\|}
    \le K  \sqrt{\frac{|\J(w)|}{\J(v)}}.
  \end{align*}
\end{lemma}

To prove Theorem~\ref{thm:necessity} we use the
following result from \cite{Goum07}, ensuring the
existence of adapted metrics for dominated splittings
over Banach bundle automorphisms and flows.

Let $\Lambda$ be a compact invariant set for a $C^1$ vector
field $X$ and let $E$ be a vector bundle over $M$.
\begin{theorem}
  \label{thm:phyp-adapted-metric}
  Suppose that $T_{\Lambda}M = F_-\oplus F_+$ is a dominated
  splitting for a linear multiplicative cocycle $A_t(x)$
  over $E$.
  There exists a neighborhood $V$ of $\Lambda$ and a
  Riemannian metric $<<\cdot,\cdot>>$ inducing a norm
  $|\cdot|$ on $E_V$ such that there exists $\lambda>0$
  satisfying for all $t>0$ and $x\in\Lambda$
  %$x\in \cap_{0\le t\le s} X_{-t}(V)$
  \begin{align*}
    % |A_t(x)\vert_{E^s_x}|\le e^{-\lambda t}
    % \qand
    |A_t(x)\vert_{F_-(x)}|\cdot\big|(A_t(x)\vert_{F_+(x)})^{-1}|
    <e^{-\lambda t}.
  \end{align*}
\end{theorem}

\begin{remark}
  A similar result holds for the existence of adapted metric
  for partially hyperbolic and for uniformly hyperbolic
  splittings. Moreover, in
    the adapted metric the bundles $F_\pm$ over
    $\Lambda$ are almost orthogonal, that is, given
    $\epsilon>0$ it is
    possible to construct such metrics so that
    $|<<v_-,v_+>>|<\epsilon$ for all $v_\pm\in F_\pm$
    with $\J(v_\pm)=\pm1$. However this property will
    not be used in what follows.
\end{remark}

We may assume, without loss of generality, that $V$ given by
Theorem~\ref{thm:phyp-adapted-metric} coincides with $U$.
Now, we use the adapted riemannian metric to define the
quadratic form on a smaller neighborhood of $\Lambda$ inside
$U$.

\subsubsection{Construction of the field of quadratic forms}
\label{sec:constr-field-quadrat}

First we choose a continuous field of orthonormal
basis (with respect to the adapted metric)
$\{e_1(x),\dots, e_s(x)\}$ of $F_-(x)$ and
$\{e_{s+1}(x),\dots, e_{s+c}(x)\}$ of $F_+(x)$ for
$x\in\Lambda$, where $s=\dim F_-$ and $c=\dim
F_+$. Then $\{e_i(x)\}_{i=1}^{s+c}$ is a basis for
$E_x$, $x\in \Lambda$.

Secondly, we consider the following quadratic forms
\begin{align*}
  \J_x(v)=\J_x\left(\sum_{i=1}^{s+c}\alpha_ie_i(x)\right)
  :=|v^+|^2-|v^-|^2
  =
  \sum_{i=s+1}^{s+c} \alpha_i^2 - \sum_{i=1}^s \alpha_j^2,
  \quad v\in E_x , x\in V,
\end{align*}
where $v^\pm\in F_\pm(x)$ are the unique orthogonal
projections on the subbundles such that $v=v^-+v^+$.
This defines a field of quadratic forms on $\Lambda$.

We note that, since $F_-\oplus F_+$ is $A_t(x)$-invariant
over $\Lambda$, and the vector field $X$ and the flow $X_t$
are $C^1$, the field of quadratic forms constructed above is 
differentiable along the flow direction, because
$F_\pm(X_t(x))=A_t(x)\cdot F_\pm(x)$ is differentiable
in $t\in\RR$ for each $x\in\Lambda$.

Clearly $F_-$ is a $\J$-negative subspace and $F_+$ is a
$\J$-positive subspace, which shows that the index of $\J$
equals $s$ and that the forms are non-degenerate.

In addition, we have strict $\J$-separation over
$\Lambda$. Indeed, $v=v^-+v^+\in C_+(x)\cup C_0(x)$ for
$x\in\Lambda$ means $|v^+|\ge |v^-|$ and the
$A_t(x)$-invariance of $F_\pm$ ensures that
$A_t(x)v=A_t(x)v^-+A_t(x)v^+$ with $A_t(x)v^\pm\in
F_\pm(X_t(x))$ and $\sqrt{\J(A_t(x)v^+)}=
|A_t(x)v^+|>e^{\lambda t}|A_t(x)
v^-|=\sqrt{|\J(A_t(x)v^-)|}$, so that $A_t(x)v\in
C_+(X_t(x))$.

We are ready to obtain the reciprocal of item 3 of
Theorem~\ref{thm:J-separated-tildeJ}.

\begin{proposition}
  \label{pr:strict-J-sep}
  If the cocycle $A_t(x)$ is strictly $\J$-separated over a
  compact $X_t$-invariant subset $\Lambda$, then there exist
  a compatible field of quadratic forms $\J_0$ and a
  function $\delta:\Lambda\to\RR$ such that $\tilde
  \J_{0,x}>\delta(x)\J_0$ for all $x\in\Lambda$.
\end{proposition}

\begin{proof}
  We have already shown that a strictly $\J$-separated
  cocycle has a dominated splitting $E=F_-\oplus F_+$ in
  Section~\ref{sec:j-separat-cocycl}.  Then we build the
  field of quadratic forms $\J_0$ according to the previous
  arguments in this section, and calculate for $v_0\in
  C_0(x), v_0=v^-+v^+$ with $v^\pm\in F_\pm(x)$ and
  $|v^\pm|=1$, for a given $x\in\Lambda$ and all $t>0$
  \begin{align}\label{eq:J0-derivative}
    \J_0(A_t(x)v_0)
    =
    |A_t(x)v^-|^2\left(\frac{|A_t(x)v^+|^2}{|A_t(x)v^-|^2}
      -1\right)
    \ge
    |A_t(x)v^-|^2\cdot(e^{2\lambda t}-1).
  \end{align}
  The derivative of the right hand side above satisfies
  \begin{align*}
    2\lambda e^{2\lambda t}|A_t(x)v^-|^2+ (e^{2\lambda
      t}-1)\partial_t|A_t(x)v^-|^2
    \xrightarrow[t\searrow0]{}2\lambda.
  \end{align*}
  Since the left hand side and the right hand side of
  (\ref{eq:J0-derivative}) have the same value at $t=0$ (we
  note that $\J_0(v_0)=0$ by the choice of $v_0$), we have
  \begin{align*}
    \tilde
    \J_x(v_0)=\partial_t\J_0(A_t(x)v_0)\mid_{t=0}\ge2\lambda>0,
    \quad x\in\Lambda.
  \end{align*}
  Thus, $\tilde \J_x(v_0)>0$ for $\vec0\neq v_0\in C_0(x)$
  which implies by Lemma~\ref{le:kuhne} that $\tilde
  \J_x>\delta(x)\J_0$ for some real function $\delta(x)$.
  Finally, the quadratic forms $\J_0$ and $\J$ are
  compatible.
\end{proof}

\subsubsection{Continuous/smooth extension to a neighborhood}
\label{sec:contin-extens-neighb}

We recall that the adapted Riemannian metric
  is defined on a neighborhood $V$ of $\Lambda$. We
  can write $\J_x(v)=<<J_x(v),v>>_x$ for all $v\in T_xM$,
  $x\in\Lambda$, where $J_x:T_xM\circlearrowleft$ is a
  self-adjoint operator.  This operator can be represented
  by a matrix (with respect to the basis adapted to $\J_x$)
  whose entries are continuous functions of $x\in\Lambda$.

These functions can be extended to
  continuous functions on $V$ yielding a continuous
  extension $\hat\J$ of $\J$. We recall that the field $\J$
  is differentiable along the flow direction. Thus $\hat\J$
  remains differentiable along the flow direction over the
  points of $\Lambda$.

Finally, these functions can then be $C^1$ regularized so
that they become $\epsilon$-$C^0$-approximated by $C^1$
functions on $V$. We obtain in this way a smooth extension $\bar\J$ of $\J$ to
a neighborhood of $\Lambda$ in such a way that $\bar\J$ is
automatically $C^1$ close to $\J$ over orbits of the flow on
$\Lambda$. This means that, given $\epsilon>0$, we can find
$\bar \J$ such that
\begin{itemize}
\item $|\hat\J_y(v)-\bar\J_y(v)|<\epsilon$ for all $v\in E_y,
  y\in V$ ($C^0$-closeness on $V$);
\item
  $|\partial_t\hat\J_{X_t(x)}(A_t(x)v)-\partial_t\bar\J_{X_t(x)}(A_t(x)v)|<\epsilon$
  for all $v\in E_x, x\in \Lambda$ and $t\in\RR$.
\end{itemize}

\begin{remark}
  \label{rmk:equivalent-qforms}
    We note that $F_\pm(x)$ are subspaces
    with the same sign for both $\J$ and $\bar\J$. Hence
    $\pm\J$ and $\pm\bar\J$ define inner products in these
    finite dimensional vector spaces, thus we can find
    $C_\pm(x)$ such that $C_\pm(x)^{-1}|\J\mid_{F_\pm(x)}|
    \le|\bar\J\mid_{F_\pm(x)}|\le
    C_\pm(x)|\J\mid_{F_\pm(x)}|$. This ensures that $\J$ and
    $\bar\J$ are equivalent forms over $\Lambda$: since $\J$
    and $\bar\J$ are continuous on $\Lambda$ we just have to
    take $C=\max\{C_\pm(x): x\in\Lambda\}$. Moreover we
  also have that the Riemannian norm $\|\cdot\|$ of $M$
and the adapted norm $|\cdot|$ are also equivalent: we
can assume that $C^{-1}|\cdot|\le \|\cdot\| \le
C|\cdot|$.
\end{remark}

Using the compatibility between $\J$ and
  $\bar\J$ we obtain for $v_\pm\in F_\pm(x)$
\begin{align*}
  |\bar\J(A_t(x)v_-)|
  &\le
  C|\J(A_t(x)v_-)|
  \le
  C\cdot K\|A_t(x)v_-\|^2
  \le
  CK \cdot C^2|A_t(x)v_-|^2
  \\
  &\le
   C^3K\cdot e^{-2\lambda t}|A_t(x)v_+|^2
  \le
  C^4K e^{-2\lambda t}\|A_t(x)v_+\|^2
  \\
  &\le
  C^4K e^{-2\lambda t}\cdot K
  \cdot\J(A_t(x)v_+)
  \le
  C^4K^2 e^{-2\lambda t}
  \cdot C \bar \J(A_t(x)v_+)
  \\
  &\le
  \kappa e^{-2\lambda t}\bar \J(A_t(x)v_+)
\end{align*}
for all $t>0$, where we used Lemma~\ref{le:ratio} together
with Theorem~\ref{thm:phyp-adapted-metric}.

Therefore we have the relations in the statement of
Theorem~\ref{thm:necessity} if we set
$\omega=2\lambda$.

\subsubsection{Strict separation for the extension/smooth approximation}
\label{sec:strict-separat-exten}

We now show that $Y$ is strictly $\bar\J$-separated on $V$
for every vector field $Y$ in a neighborhood $\V$ of $X$ and
for every multiplicative cocycle $B_t(x)$ over $Y$ which is
$C^0$ close to $A_t(x)$.

We start by observing that $\bar\J$ is differentiable along
the flow direction. Then we note that, from
Proposition~\ref{pr:strict-J-sep}
\begin{align*}
  \iota:=\inf\{ \widetilde \J_x-\delta(x) \J_x : x\in\Lambda\} >0
\end{align*}
and recall that $\widetilde J_x= J_x\cdot
D(x)+D(x)^*\cdot J_x$.  Hence, by choosing $V$ sufficiently
small around $\Lambda$, we obtain
\begin{align*}
  \hat\iota=\inf\{\widetilde{\hat \J_y}-\delta(y)\hat \J_y : y\in V\}
  \ge\frac\iota2 >0,
\end{align*}
since $\hat\J$ is an extension of $\J$ on $\Lambda$, and the
function $\delta$ is defined by $\hat J$ and $D(x)$
according to Remark~\ref{rmk:axcont}. Finally, by taking a
sufficiently small $\epsilon>0$ in the choice of the $C^1$
approximation $\bar J$, we also get
\begin{align*}
  \bar\iota=\inf\{\widetilde{\bar \J}_y-\delta(y)\bar \J_y :
  y\in V\}\ge\frac{\hat\iota}2\ge\frac\iota4>0.
\end{align*}
From Theorem~\ref{thm:J-separated-tildeJ} and
Proposition~\ref{pr:strict-J-sep}, we know that this is a
necessary and sufficient condition for strict
$\bar\J$-separation of $A_t(x)$ over $V$.

\subsubsection{Strict separation for nearby flows/cocycles}
\label{sec:strict-separat-nearb}

Given a vector field $X$ on $M$ and a linear multiplicative
cocycle $A_t(x)$ on a vector bundle $E$ over $M$, for a
$C^1$ close vector field $Y$ and a $C^0$ close cocycle
$B_t(x)$ over $Y$, the infinitesimal generator $D_{B,Y}(x)$ of $B_t(x)$
will be a linear map close to the infinitesimal generator
$D(x)$ of $A_t(x)$ at $x$. That is, given $\epsilon>0$ we
can find a $C^1$ neighborhood $\V$ of $X$ and a $C^0$
neighborhood $\W$ of the cocycle $A$ such that
\begin{align*}
  (Y,B)\in\V\times\W\implies \|D_{B,Y}(x)-D(x)\| < \epsilon,
  \quad x\in M.
\end{align*}
Hence, since $\delta$ also depends continuously on the
infinitesimal generator, we obtain
\begin{align*}
  \tilde\iota=\inf\{\widetilde{\bar\J}_y -
  \delta_{B,Y}(y)\bar\J_y : y\in V, Y\in\V, B\in\W\} \ge
  \frac{\bar\iota}2 >0.
\end{align*}
This shows that we have strict $\bar\J$ separation for all
nearby cocycles over all $C^1$-close enough vector fields
over the same neighborhood $V$ of the original invariant
attracting set $\Lambda$.

Finally, to obtain the inequalities of the statement of
Theorem~\ref{thm:necessity}, since we have strict
$\bar\J$-separation for $B_t(x)$, we also have a dominated
splitting $E_{\Lambda_Y}=F^B_-\oplus F^B_+$ over
$\Lambda_Y$ whose subbundles have the same sign as the
original $F_\pm$ subbundles of $E_\Lambda$ for
$A_t(x)$. We can then repeat the arguments leading to
the constant $\kappa$, which depends continuously on
$Y$.

This completes the proof of Theorem~\ref{thm:necessity}.

%%%%%%%%%%%%%%%%%%%%%%%%%%%%%%%%%%%%%%%%%%%%%%%%%%%%%%%%%%%

\subsection{Characterization of the splitting through the
  function $\delta$}
\label{sec:charact-splitt-throu}

We now use the area under the function $\delta$ to
characterize different dominated splittings that may arise
in our setting.

\begin{theorem}
  \label{thm:char-dom-split}
  Let $\Gamma$ be a compact invariant set for $X_t$
  admitting a dominated splitting $E_\Gamma= F_-\oplus F_+$
  for $A_t(x)$, a linear multiplicative cocycle over
  $\Gamma$ with values in $E$. Let $\J$ be a $C^1$ field of
  indefinite non-degenerate quadratic forms such that
  $A_t(x)$ is strictly $\J$-separated, admitting a function
  $\delta:\Gamma\to\RR$ as given in
  Theorem~\ref{thm:J-separated-tildeJ}. Then
  \begin{enumerate}
  \item If
    $\Delta_s^t(x)\xrightarrow[(t-s)\to+\infty]{}-\infty$
    for all $x\in\Gamma$, then $F_-$ is a uniformly
    contracted subbundle.
% $F_-\oplus F_+$ is partially hyperbolic with $F_-$
%     not uniformly contracting and $F_+$ uniformly expanding
%     if, and only if,
%     $$\Delta_s^t(x)\xrightarrow[(t-s)\to+\infty]{} +\infty$$
%     for all $x\in\Gamma$ and $\Delta_s^t(x)$ is bounded from
%     below as $(t-s)\to +\infty$ for some $x \in \Lambda$.
  \item If $\Delta_s^t(x)\xrightarrow[(t-s)\to+\infty]{} +\infty$
for all $x\in\Gamma$, then $F_+$ is a uniformly expanding subbundle.
% $F_-\oplus F_+$ is partially hyperbolic with $F_-$
%     uniformly contracting and $F_+$ not uniformly expanding
%     if, and only if,
%     $$\Delta_s^t(x)\xrightarrow[(t-s)\to+\infty]{}-\infty$$
%     for all $x\in\Gamma$ and $\Delta_s^t(x)$ is bounded from
%     below as $(t-s)\to -\infty$ for some $x \in \Lambda$.
  \item $F_-\oplus F_+$ is hyperbolic (that is, $F_-$ is
    uniformly contracted and $F_+$ is uniformly expanded)
    if, and only if, there exists a compatible field
    of quadratic forms $\J_0$ in a neighborhood of $\Gamma$ such
    that $\J_0'(v)>0$ for all $v\in E_x$ and all $x\in\Gamma$.
  \end{enumerate}
\end{theorem}

Above we write $\J'(v)= <\tilde J_x v, v>$ where
$\tilde J_x$ is given in Proposition~\ref{pr:J-separated}.

In the proof we use the following useful equivalence.

\begin{lemma}
  \label{le:unif-behav-lim0}
  Let $F\subset E$ be a continuous $A_t(x)$-invariant
  subbundle of the finite dimensional vector bundle $E$
  with compact base $\Lambda$. Then, there are
  constants $K,\omega>0$ satisfying for $\vec0\neq v\in
  F_x, x\in\Lambda, t>0$
  \begin{align*}
    \|A_t(x)v\|\le K e^{-\omega t} \|v\| \quad
    (\|A_{-t}(x)v\|\le K e^{-\omega t} \|v\|, \text{
      respectively})
  \end{align*}
  if, and only if, for every $x\in\Lambda$ and $\vec0\neq
  v\in E_x$
  \begin{align*}
    \lim\limits_{t\to+\infty}\|A_t(x)v\| = 0 \quad
    (\lim\limits_{t\to+\infty}\|A_{-t}(x)v\| = 0, \text{
      respectively}).
  \end{align*}
\end{lemma}

\begin{proof}
  See e.g. \cite{Man82}.
\end{proof}

\begin{proof}[Proof of Theorem~\ref{thm:char-dom-split}]
  \begin{enumerate}
  \item If
    $\Delta_0^t(x)\xrightarrow[t\to+\infty]{}-\infty$, then
    from item (2) of Theorem~\ref{thm:J-separated-tildeJ} we
    get $ \frac{\J(A_t(x)v)}{\J(v)}\le
    e^{\Delta_0^t(x)}\xrightarrow[t\to+\infty]{}0 $ for all
    $x\in\Gamma$ and $v\in F_-(x)$. So $F_-$ in uniformly
    contracted, by Lemmas~\ref{le:ratio}
    and~\ref{le:unif-behav-lim0}.
  \item 
    If $\Delta_s^t(x)\xrightarrow[(t-s)\to+\infty]{}
    +\infty$, then analogously $
    \frac{\J(A_t(x)v)}{\J(v)}\ge
    e^{\Delta_0^t(x)}\xrightarrow[t\to+\infty]{}+\infty $
    for all $x\in\Gamma$ and $v\in F_+(x)$. So $F_+$ in
    uniformly contracted, again by Lemmas~\ref{le:ratio}
    and~\ref{le:unif-behav-lim0}.
  \end{enumerate}

Now let us assume that $F_-\oplus F_+$ is a uniformly
hyperbolic splitting and take $\J_0$ the field of quadratic
forms provided by Theorem~\ref{thm:phyp-adapted-metric},
which is compatible with $\J$.

For $v=v_-+v_+\in E_x$ with $v_\pm\in
F_\pm(x)$ and $\J(v)>0$ (note that the difference below is
positive for small $|t|$)
  \begin{align}\label{eq:J0-derivative}
    \J_0(A_t(x)v)
    &=
    |A_t(x)v_+|^2
    \left(1-\frac{|A_t(x)v_-|^2}{|A_t(x)v_+|^2}\right)
    \ge
    |A_t(x)v_+|^2\left(1-e^{-2\lambda
        t}\frac{|v_-|^2}{|v_+|^2}\right).
  \end{align}
  The derivative of the right hand side above satisfies
  \begin{align*}
    2\lambda e^{2\lambda t}
    |A_t(x)v_+|^2\frac{|v_-|^2}{|v_+|^2}+
    \left(|v_+|^2-e^{2\lambda t}|v_-|^2\right)
    \frac{\partial_t|A_t(x)v_+|^2}{|v_+|^2}
    \xrightarrow[t\searrow0]{}
    2\lambda|v_-|^2+\J_0(v)\frac{\partial_t|A_t(x)v_+|^2}{|v_+|^2}.
  \end{align*}
  Since the left hand side and the right hand side of
  (\ref{eq:J0-derivative}) have the same value at $t=0$, the
  limit above is a lower bound for
  $\partial_t\J_0(A_t(x)v)\mid_{t=0}$. Because
  $\J_0(A_t(x)v_+)\ge e^{2\lambda t}\J_0(v)$ and $\J_0(v)>0$
  \begin{align*}
    \J_0'(v)=\partial_t\J_0(A_t(x)v)\mid_{t=0}
    \ge
    2\lambda|v_-|^2+2\lambda\J_0(v)=2\lambda|v_+|^2>0.
  \end{align*}
  Moreover, since for any non-zero vector $v_0=v_-+v_+$ in
  $C_0(x)$ we can make an arbitrarily small perturbation to
  $v_-$, keeping $v_+$, so that $\J_0(\tilde v_-+v_+)>0$, we
  obtain $\J_0'(\tilde v_-+v_+)\ge2\lambda|v_+|^2$ and so
  $\J_0'(v_0)>0$ for all non-zero $v_0$ in $C_0(x)$.
  Now for $\J_0(v)<0$ we have
  \begin{align}\label{eq:J0-derivative-neg}
    \J_0(A_t(x)v)
    &=
    |A_t(x)v_-|^2
    \left(\frac{|A_t(x)v_+|^2}{|A_t(x)v_-|^2}-1\right)
    \ge
    |A_t(x)v_-|^2\left(e^{2\lambda
        t}\frac{|v_+|^2}{|v_-|^2}-1\right)\nonumber
    \\
    &\ge
    e^{-2\lambda t}|v_-|^2\left(e^{2\lambda
        t}\frac{|v_+|^2}{|v_-|^2}-1\right)
    =|v_+|^2-e^{-2\lambda t}|v_-|^2
  \end{align}
  where we used domination in the first inequality and
  $\J_0(v)<0$ and $|t|$ small in the second inequality.  Hence
  $\J_0'(v)\ge2\lambda|v_-|^2>0$.  This completes the proof of
  the sufficient condition of item (3).

  Reciprocally, let us assume that $\J'$ is a positive
  definite quadratic form. Hence $\J'$ is an inner product
  on a finite dimensional vector bundle $E_\Gamma$ with
  compact base, and so there exists $\kappa>0$ such that
  $\J'\ge \kappa|\cdot|^2$ and
  $\kappa|\J|\le|\cdot|^2$. Thus $\J'(v)\ge \kappa^2|\J(v)|$
  for all $v\in E$, which implies $\J(A_t(x)v_+)\ge
  e^{\kappa^2 t}\J(v_+)$ for $\vec0\neq v_+\in F_+(x)$; and
  $|\J(A_t(x)v_-)|\le e^{-\kappa^2 t}|\J(v_-)|$ for all
  $\vec0\neq v_-\in F_-(x)$. This shows, from the comparison
  results given in Lemma~\ref{le:ratio}, that $F_-\oplus
  F_+$ is a uniformly hyperbolic splitting of $E_\Gamma$,
  and completes the proof of
  Theorem~\ref{thm:char-dom-split}.
\end{proof}

%%%%%%%%%%%%%%%%%%%%%%%%%%%%%%%%%%%%%%%%%%%%%%%%%%%%%%%%%%%%%%%

\section{Partial hyperbolicity - Proof of Theorem \ref{mthm:Jseparated-parthyp}}
\label{sec:strict-j-separat}

Now we prove Theorem~\ref{mthm:Jseparated-parthyp}. We show
that strict $\J$-separation of a $\J$-non-negative flow
$X_t$ on a trapping region $U$ implies the existence of a
dominated splitting and that the dominating bundle (the one
with the weakest contraction) is necessarily uniformly
contracting. That is, we have in fact a partially hyperbolic
splitting.

The strategy is to consider the derivative cocycle $DX_t$ of
the smooth flow $X_t$ in the place of $A_t(x)$ and use the
results of Section~\ref{sec:propert-j-separat}. In this
setting we have that the infinitesimal generator is given by
$D(x)=DX(x)$ the spatial derivative of the vector field
$X$. Since the direction of the flow $E^X_x:=\RR\cdot X(x) =
\{ s\cdot X(x): s\in\RR\}$ is $DX_t$-invariant for all $t\in\RR$, if $U$ is a
trapping region where $X_t$ is
$\J$-separated and $\J(X(x))\ge0$ for some $x\in U$, then
$\J(DX_t(X(x)))\ge0$ for all $t>0$ and this function is
bounded. % This implies,
% from~\eqref{eq:Jexponential-inequality}, that
% $\Delta(x,t)=\int_0^t \delta\circ X^s(x)\,ds$ must be a
% bounded function for $t>0$.

% \begin{remark}\label{rmk:J-non-negative}
%   In the same setting, if $\J(X(x))\le0$, then $\Delta(x,t)$
%   is bounded for $t<0$.  This restricts the possible
%   functions $\delta(x)$ for $\J$-separated flows.
% \end{remark}

We recall Lemma~\ref{le:trapping} and deduce the following.

\begin{corollary}
  \label{cor:persistence}
  Let $X_t$ be strictly $\J$-separated on $U$.
  Then there exist a neighborhood $\U$ of $X$ in
  $\Mundo$ and a neighborhood $V$ of $\Lambda(U)$ such
  that $V$ is a trapping region for every $Y\in\U$ and
  each $Y\in\U$ is strictly $\J$-separated in $V$.
\end{corollary}

\begin{proof}
  The assumption implies that $\tilde \J_x= \tilde \J_x^X >
  \alpha(x)\J$ for all $x\in U$. Let $\U$ and $V$ be the
  neighborhoods of $X$ and $\Lambda$ given by
  Lemma~\ref{le:trapping}. Let also $Y\in\U$ be fixed.

  Writing $\tilde J_x^Y:= J\cdot DY(x)+DY(x)^*\cdot J$, we can make
  the norm $\|\tilde \J_x - \tilde \J_x^Y\|$ smaller than
  \begin{align*}
    \frac12\min\left\{ \inf_{v\in C_+(x)}
    \frac{\langle \tilde J_x v,v \rangle}{\langle J
      v,v\rangle}
    - \alpha(x) : x \in \overline{U_0}\right\}
  \end{align*}
  for all $x\in V$ by shrinking $\U$ if needed.  This
  ensures that there exists $\beta:\overline V\to\RR$ such
  that $\tilde \J_x^Y>\beta(x) \J$ for all $x\in \overline V$, so $Y$
  is strictly $\J$-separated on $V$.
\end{proof}

From Section~\ref{sec:j-separat-cocycl} we know that there
exists a continuous dominated splitting $F^Y_-(x)\oplus
F^Y_+(x)$ of $T_xM$ for $x\in\Lambda_Y(U)$, with respect to
$Y_t$ for all $Y\in\U$.

The strict $\J$-separation on $U$ for $X_t$ implies that
any invariant subbundle of $T_xM$ along an orbit of the flow
$X_t$ must be contained in $F_\pm(x)$. In particular, the
characteristic space corresponding to the flow direction is
contained in $F_+(x)$.

\begin{lemma}
  \label{le:flow-center}
  Let $\Lambda$ be a compact invariant set for a flow $X$ of
  a $C^1$ vector field $X$ on $M$.

  \begin{enumerate}
  \item Given a continuous splitting $T_\Lambda M =
    E\oplus F$ such that $E$ is uniformly contracted,
    then $X(x)\in F_x$ for all $x\in \Lambda$.
  \item Assuming that $\Lambda$ is
      non-trivial and has a continuous \emph{and dominated}
      splitting $T_\Lambda M = E\oplus F$ such that $X(x)\in
      F_x$ for all $x\in\Lambda$, then $E$ is a uniformly
      contracted subbundle.
  \end{enumerate}
\end{lemma}

\begin{proof} For the first item, we denote by
  $\pi(E_x):T_x M\to E_x$ the projection on $E_x$
  parallel to $F_x$ at $T_x M$, and likewise
  $\pi(F_x):T_x M\to F_x$ is the projection on $F_x$
  parallel to $E_x$. We note that for $x\in\Lambda$
  \begin{align*}
    X(x)=\pi(E_x)\cdot X(x) + \pi(F_x)\cdot X(x)
  \end{align*}
  and for $t\in\R$, by linearity of $DX_t$ and
  $DX_t$-invariance of the splitting $ E\oplus F$
  \begin{align*}
    DX_{t}\cdot X(x)
    &=
    DX_{t}\cdot\pi(E_x)\cdot X(x)
    +
    DX_{t}\cdot\pi(F_x)\cdot X(x)
    \\
    &=
    \pi(E_{X_{t}(x)})\cdot DX_{t} \cdot X(x)
    + \pi(F_{X_{t}(x)}) \cdot DX_{t} \cdot X(x)
  \end{align*}
  Let $z$ be a limit point of the negative orbit of $x$,
  that is, we assume that there is a strictly increasing
  sequence $t_n\to+\infty$ such that
  $\lim\limits_{n\to+\infty}x_{n}:=\lim\limits_{n\to+\infty}X_{-t_n}(x)
  = z$. Then $z\in\Lambda$ and, \emph{if} $\pi(E_x)\cdot
  X(x)\neq\vec0$ we get
  \begin{align}
    \lim\limits_{n\to+\infty} DX_{-t_n}\cdot X(x)
    &=
    \lim\limits_{n\to+\infty} X(x_n) = X(z) \quad\text{but
      also}
    \nonumber
    \\
    \|DX_{-t_n}\cdot\pi(E_x)\cdot X(x)\|
    &\ge
    c e^{\lambda t_n}\|\pi(E_x)\cdot
    X(x)\|\xrightarrow[n\to+\infty]{}+\infty.
    \label{eq:expand}
  \end{align}
  This is possible only if the angle between
  $E_{x_n}$ and $F_{x_n}$ tends to zero when
  $n\to+\infty$.

  Indeed, using the Riemannian metric on $T_y M$, the angle
  $\alpha(y)=\alpha(E_y,F_y)$ between $E_y$ and $F_y$ is
  related to the norm of $\pi(E_y)$ as follows:
  $\|\pi(E_y)\|=1/\sin(\alpha(y))$. Therefore
  \begin{align*}
    \|DX_{-t_n}\cdot\pi(E_x)\cdot X(x)\|
    &=
    \|\pi(E_{x_n})\cdot DX_{-t_n} \cdot X(x) \|
    \\
    &\le
    \frac1{\sin(\alpha(x_n))}
    \cdot
    \|DX_{-t_n} \cdot X(x) \|
    \\
    &=
    \frac1{\sin(\alpha(x_n))}
    \cdot \|X(x_n)\|
  \end{align*}
  for all $n\ge1$. Hence, if the sequence \eqref{eq:expand}
  is unbounded, then
  $\lim\limits_{n\to+\infty} \alpha(X_{-t_n}(x)) = 0$.

  However, since the splitting $E\oplus F$ is continuous
  over the compact $\Lambda$, the angle $\alpha(y)$ is a
  continuous and positive function of $y\in\Lambda$, and
  thus must have a positive minimum in $\Lambda$. This
  contradiction shows that $\pi(E_x)\cdot X(x)$ is always
  the zero vector and so $X(x)\in F_x$ for all
  $x\in\Lambda$.

  This completes the proof of the first item.

    For the second item, fix $x\in\Lambda$
    with $X(x)\neq\vec0$, take $v\in E_x$ and use the
    definition of $(K,\lambda)$-domination to obtain for
    each $t>0$
  \begin{align*}
    K e^{-\lambda t}\ge \frac{\|DX_t\cdot v\|}{\|DX_t\cdot X(x)\|}
    =
    \frac{\|DX_t\cdot v\|}{\|X(X_t(x))\|}
    \ge
    \frac1C\|DX_t\cdot v\|
  \end{align*}
  where $C=\sup\{\|X(y):y\in\Lambda\}$ is a positive
  number. For $\sigma\in\Lambda$ such that
  $X(\sigma)=\vec0$, we fix $T>0$ such that $CK e^{-\lambda
    T}<1/2$ and, since $\Lambda$ is non-trivial, we can find
  a sequence $x_n\to\sigma$ of regular points of
  $\Lambda$. The continuity of the derivative cocycle
  ensures $1/2\ge\|DX_T\mid E_\sigma\|=\lim_{n\to+\infty}
  \|DX_T\mid E_{x_n}\|$ and so $E_\sigma$ is also a
  uniformly contracted subspace.
\end{proof}

Since $X$ and $\Lambda=\Lambda_X$ are in the conditions of the
second item of the previous lemma, we have proved partial
hyperbolicity for $T_{\Lambda}M=F_-\oplus F_+$.

At this point, we have concluded the proof of sufficiency in
the statement Theorem~\ref{mthm:Jseparated-parthyp}.

The necessary part of the statement of
Theorem~\ref{mthm:Jseparated-parthyp} is a simple
consequence of Theorem~\ref{thm:necessity} applied to the
cocycle $DX_t$ acting on the vector bundle $T_UM$, the
tangent bundle on the trapping region $U$.

This completes the proof of
Theorem~\ref{mthm:Jseparated-parthyp}.

\subsection{Uniform Hyperbolicity}
\label{sec:uniform-hyperb}

By using Theorem \ref{mthm:Jseparated-parthyp}, we are able
to present the proof of the Corollaries
\ref{mcor:J-separation-hyperbolicity} and
\ref{mcor:incompress-anosov}.

\begin{proof}[Proof of Corollary \ref{mcor:J-separation-hyperbolicity}]
  Let $X \in \Mundo$ and $\Lambda$ be the maximal invariant
  set of a trapping region $U$ for $X$.  Consider $\J, \G$
  two differentiable fields of non-degenerated quadratic
  forms on $U$ with constant indices $s$ and $n - s- 1$,
  respectively, where $n = \dim M$ and $s < n - 2$.  Since
  $\Lambda \cap \sing(X) = \emptyset$, the flow direction
  $X(x)$ is non-zero for all point $x \in \Lambda$ and
  generates an invariant line bundle $E^X$ over
  $\Lambda$. % By compactness of $\Lambda$, the norm of $X(x)$
  % is bounded below away from zero and also bounded above.

  On the one hand, by Remark \ref{rmk:strictly-J-separated},
  $- X$ is a non-negative strictly separated with respect to
  $- \G$ on $\Lambda$. Then Theorem
  \ref{mthm:Jseparated-parthyp} implies that there is a
  partially hyperbolic splitting $T_\Lambda M=E^{cs} \oplus
  E^{u}$ with the subbundle $E^u$ uniformly expanding, and
  dimensions $\dim E^{cs}=s+1$ and $\dim E^u=n-s-1$. Thus,
  by Lemma \ref{le:flow-center}, the flow direction $E^X$ is
  contained in $E^{cs}$.

  On the other hand, with analogous arguments, we prove
  that, for $\J$, $\Lambda$ has a partially hyperbolic
  splitting $T_\Lambda M=E^s \oplus E^{cu}$, with $E^s$
  uniformly contracting, and so $E^X \subset E^{cu}$, with
  dimensions $\dim E^s=s$ and $\dim E^{cu}=n-s$.

  Moreover we clearly have $E^s\cap E^{cu}=\{0\}$ and
  $E^u\cap E^{cs}=\{0\}$.  Hence we have the following
  dominated splittings
$$
(E^{s}\oplus E^{X}) \oplus E^{u} = E^{s}\oplus (E^{X} \oplus
E^{u}) = T_{\Lambda}M,
$$
with $\dim E^{cs} = \dim (E^{s}\oplus E^X)$ and $\dim E^{cu}
= \dim (E^{s}\oplus E^X)$. By uniqueness of dominated
splittings with the same dimensions we obtain
$E^{cu}=E^X\oplus E^u$ and $E^{cs}=E^s\oplus E^X$, and the
splitting of $T_\Lambda M$ above is a hyperbolic splitting.
\end{proof}

\begin{proof}[Proof of Corollary \ref{mcor:incompress-anosov}]
  Consider $M$ a closed Riemannian manifold with dimension
  $n \geq 3$ and $X \in \Mundo$ a incompressible vector
  field.

  Let $\J$ be a field of non-degenerate quadratic forms on
  $M$, with constant index $\indi (\J)=\dim (M)-2$, such
  that $X_t$ is a non-negative $\J$-separated flow.

  By Theorem \ref{mthm:Jseparated-parthyp} and the
  hyphotesis on $\indi (\J)$, we obtain a partially
  hyperbolic splitting $TM = E \oplus F$, with $\dim E
  =\dim(M) - 2$ and $\dim F=2$ and $E$ uniformly
  contracted. Hence, as the flow is volume-preserving, the
  area along the two-dimensional direction $F$ is
  expanded. Indeed, the angle $\theta(E_x,F_x)$ between the
  subbundles is uniformly bounded away from zero (by
  domination of the splitting; see~\cite{Newhouse2004}) and
  so
  \begin{align*}
    1=|\det DX_t(x)| = |\det DX_t\mid_{E_x}|\cdot|\det
    DX_t\mid_{F_x}|\cdot \sin \theta(E_{X_t(x)} ,F_{X_t(x)})
   \end{align*}
  which for $t<0$ ensures that
  \begin{align*}
    |\det DX_t\mid_{F_x}|\le \sin\theta_0
    \cdot|\det DX_t\mid_{E_x}|^{-1}\xrightarrow[t\to-\infty]{}0.
  \end{align*}
  Thus, $M$ is a singular-hyperbolic set for $X$. Moreover,
  there can be no singularities, since they cannot be in the
  interior of a singular-hyperbolic set; see Doering
  \cite{Do87} and Morales-Pacifico-Pujals \cite{MPP99} in
  dimension three; Li-Gan-Wen~\cite{LGW05} and Vivier
  \cite{Viv03} for higher dimensions; or \cite[Chapter
  4]{AraPac2010}. Hence $M$ is a singular-hyperbolic set for
  $X_t$ without singularities. Therefore $X$ is an Anosov
  flow.
\end{proof}

\section{Sectional hyperbolicity - proof of Theorem
  \ref{mthm:2-sec-exp-J-monot}}
\label{sec:approach-via-linear}

Here we prove Theorem~\ref{mthm:2-sec-exp-J-monot}.  We
assume that $X$ is a $C^1$ vector field in an open trapping
region $U$ with a smooth field of non-degenerate
quadratic forms $\J$ such that $X$ is non-negative and strictly
$\J$-separated, and the linear Poincar\'e flow on any
compact invariant subset $\Gamma$ of
$\Lambda^*_X(U):=\Lambda_X(U)\setminus \sing(X)$ is strictly
$\J_0$-monotone, for some field of quadratic forms $\J_0$
equivalent to $\J$.

We show that, in this setting, the linear Poincar\'e flow of
$X$ on $\Gamma$ has a hyperbolic splitting. After that, as a
consequence, we show that either there are no singularities
in $\Lambda$ and then $\Lambda$ is a hyperbolic attracting
set; or, otherwise, $\Lambda$ is a sectional hyperbolic
attracting set, as long as the singularities are sectional
hyperbolic with index compatible with the index of the
attracting set.

\subsection{Strict $\J$-monotonicity for the linear Poincar\'e
  flow and hyperbolicity}
\label{sec:j-monoton-linear-1}

Strict $\J$-monotonicity is clearly stronger than strict
$\J$-separation, so that on a compact invariant subset
$\Gamma$ of $\Lambda^*_X(U)$ the linear Poincar\'e flow
$P^t$ admits a dominated splitting $N^s\oplus N^u$ of $N$
over $\Gamma$. But with strict $\J$-monotonicity we can say
more.

Consider $X\in\Mundo$ with a trapping region $U$,
$\Lambda_X(U)$ its attracting set and a smooth field
of non-degenerate quadratic forms $\J$ on $U$.
\begin{proposition}
  \label{pr:J-monot-hyperb}
  If $X_t$ is non-negative strictly $\J$-separated on
  $\Lambda_X(U)$ and the associated linear Poincar\'e flow
  $P^t$ over any compact invariant subset $\Gamma$ of
  $\Lambda_X(U)^*$ is strictly $\J_0$-monotone for some
  field of quadratic forms $\J_0$ on $T_\Gamma M$
  equivalent to $\J$, then $\Gamma$ is a hyperbolic set for
  $P^t$.
\end{proposition}

\begin{proof}
  The property $\partial_t\J_0(P^tv)\mid_{t=0}>0$ for all
  $v\in N_x$, $x\in\Gamma$ is equivalent to say that the
  quadratic form $\partial_t \J_{0,x}\vert_{N_x}$ is positive
  definite for all $x\in\Gamma$.  This implies the existence
  of a function $\alpha_1:U\to(0,+\infty)$ such that
  \begin{align*}
    \partial_t \J_0(P^tv)\mid_{t=0} > \alpha_1(x) \cdot
    \|v\|^2>0,\quad x\in \Gamma, v\in N_x, v\neq\vec0.
  \end{align*}
  Since $\Gamma$ is compact, the smoothness of $\J_0$ ensures
  the existence of $\alpha_2>0$ such that
  $|\J_0(v)|\le\alpha_2\|v\|^2$ for all $v\in N_x$,
  $x\in\Gamma$. Hence we obtain
  \begin{align*}
    \partial_t \J_0(P^tv)\mid_{t=0} \ge \alpha_1(x) \cdot
    \|v\|^2 \ge \frac{\alpha_1(x)}{\alpha_2} |\J_0(v)|
  \end{align*}
  where $\alpha_1(x)>0$ for all $x\in\Gamma$.
  Therefore we have
  \begin{align*}
    \log\frac{\J_0(P^tv)}{\J_0(v)} &\ge \int_0^t
    \frac{\alpha_1(X^s(x))}{\alpha_2}\,ds =:H(x,t) \quad
    \text{for $\J_0$-positive vectors $v$}; \qand
    \\
    \log\frac{\J_0(P^tv)}{\J_0(v)} &\le -\int_0^t
    \frac{\alpha_1(X^s(x))}{\alpha_2}\,ds = -H(x,t) \quad
    \text{for $\J_0$-negative vectors $v$}.
  \end{align*}
  From Lemma~\ref{le:ratio} we can compare $|\J_0|$ with the
  square of the Riemannian norm, so all that is left to do
  is to prove that $H(x,t)$ is unbounded for $t>0$ and each
  $x\in\Gamma$.
  \begin{lemma}
    \label{le:A-unbound}
    For every point $x$ in a compact invariant subset
    $\Gamma\subset\Lambda_X(U)^*$, we have
    $\lim\limits_{t\to+\infty}H(x,t)=+\infty$.
  \end{lemma}

\begin{proof}[Proof of Lemma~\ref{le:A-unbound}]
  If there are no singularities in $\Lambda$, then
  $\Lambda_X(U)^*=\Lambda_X(U)$ is compact and so there
  exists $\alpha_0>0$ such that $\alpha_1(x)\ge\alpha_0$ for
  all $x\in\Lambda$. This clearly implies the statement of
  the lemma in this case.

  Let $S=\Lambda\cap S(X)$ be the set of finitely many
  singularities of $X$ in $\Lambda$; we recall that we are
  assuming that $S$ is formed by hyperbolic fixed points of
  $X_t$.  We fix $\epsilon>0$ such that
  $\{B(\sigma,\epsilon)\}_{\sigma\in S}$ is pairwise
  disjoint sets and $\Lambda\setminus
  B(S,\epsilon)\neq\emptyset$, where
  $B(S,\epsilon)=\cup_{\sigma\in S} B(\sigma,\epsilon)$.

  We have that $K:=\Lambda\setminus B(S,\epsilon)$ is
  compact and so there exists $\alpha_0>0$ such that
  $\alpha_1(x)\ge\alpha_0$ for all $x\in K$. Moreover, since
  the norm of the vector field $X$ is bounded from above in
  $\Lambda$, there exists a minimum time $T>0$ between
  consecutive visits of any orbit of $x\in\Lambda\setminus
  W^s(S)$ to $B(S,\epsilon)$. That is, for
  $x\in\Lambda\setminus W^s(S)$, if we define sequences
  $t_1<s_1<t_2<s_2< \dots$ such that
  $X_{(t_i,s_i)}(x)\subset B(S,\epsilon)$ and
  $X_{[s_i,t_{i+1}]}(x)\subset K$, then $t_{i+1}-s_i>T$,
  $i\ge1$.

  Since $x\in\Gamma$ and $\Gamma\cap S=\emptyset$, we have
  $\omega_X(x)\cap B(S,\epsilon)=\emptyset$ for some small
  $\epsilon>0$ dependent on $\Gamma$ only, in which case the
  sequence above terminates at some $s_l$ for $l\ge1$.

  Hence we can write, for all $t\ge s_{l}$ and $t<t_{l+1}$
  (if $t_{l+1}$ does not exist, the last conditions is
  vacuous)
  \begin{align*}
    H(x,t)\ge \alpha_0 t_1+ \alpha_0\sum_{i=1}^{l-1}
    (t_{i+1}-s_i) + (t-s_l)\alpha_0 \ge \alpha_0 [(l-1) T +
    t_1 + (t-s_l)].
  \end{align*}
  Either the sequences are infinite, or $t_{l+1}$ does not
  exist. Hence $H(x,t)$ grows without bound when
  $t\to+\infty$.
\end{proof}

Restricting $x$ to a compact invariant subset $\Gamma$ of
$\Lambda_X(U)^*$, we obtain
\begin{align*}
  \lim\limits_{t\to-\infty}\|P^t\vert_{N^s_x}\|=+\infty
  \qand
  \lim\limits_{t\to+\infty}\|P^t\vert_{N^u_x}\|=+\infty,
  \quad x\in\Gamma.
\end{align*}
By well-known results, this ensures that $P^t$ is
hyperbolic over $\Gamma$; see e.g. \cite{Man82}. This
concludes the proof.
\end{proof}

\subsection{Sectional hyperbolicity from the linear Poincar\'e flow}
\label{sec:section-hyperb-from}

Here we prove Theorem~\ref{mthm:2-sec-exp-J-monot} through the
following results.

Let $\Lambda=\Lambda_X(U)$ be an attracting set contained in
the non-wandering set $\Omega(X)$ for a $C^1$ flow given by
a vector field $X$. We recall that
$\Lambda^*_X(U)=\Lambda_X(U)\setminus \sing(X)$.

\begin{theorem}\label{thm:2secexp-LPFhyp}
  The set $\Lambda$ is sectional-hyperbolic for $X$ if, and
  only if, there is a neighborhood $\U$ of $X$ in $\Mundo$
  such that any compact subset $\Gamma$ of $\Lambda^*_Y(U)$
  is hyperbolic of index $\indi(\Lambda)$ for the linear
  Poincar\'e flow associated to any $Y\in\U$ and each
  singularity $\sigma$ of $Y$ in the trapping region $U$ is
  sectionally hyperbolic with index $\indi(\sigma) \geq
  \indi(\Lambda)$.
\end{theorem}

\begin{proof} Let us assume that
  $\Lambda=\Lambda_X(U)\subset\Omega(X)$ is sectionally
  hyperbolic for $X$.  If $E^s\oplus E^c$ is a
  partially hyperbolic splitting of $TM$ over
  $\Lambda$, then the projections $N^s:=\Pi\cdot E^s$
  and $N^u:=\Pi\cdot E^c$ are $P^{\,t}$-invariant, and
  $N^s$ is uniformly contracted by $P^{\,t}$. Indeed,
  since the orthogonal projection does not increase
  norms, for $v\in E^s_x$ we get $\|P^t
  v\|=\|\Pi_{X_t(x)} DX_t(x)\cdot v\|\le\|DX_t(x)\cdot
  v\|$, which is uniformly contracted for $t>0$, as
  long as $X(x)\neq\vec0$.

  Moreover, the above property persists for all vector
  fields $Y$ in a small enough $C^1$ neighborhood of $X$, by
  the normal hyperbolic theory; see \cite{HPS77}.

  Now we assume, additionally, that $E^c$ is sectionally
  expanding on $\Lambda$ for $X$. This ensures that the
  continuation $E^{s,Y}\oplus E^{c,Y}$ of the partially
  hyperbolic splitting for $C^1$ close vector fields is also
  sectional hyperbolic. For otherwise, according to
  Remark~\ref{rmk:sec-exp-discrete}, for any given fixed
  $T>0$ we would obtain a sequence $Y_n$ of vector fields
  converging to $X$ in the $C^1$ topology, a sequence $x_n$
  of points in $\Lambda_{Y^n}(U)^*$ and a sequence $F_{x_n}$
  of $2$-subspaces of $E^{c,Y}_{x_n}$ such that
  $|\det(DY^n_T\vert_{F_{x_n}})|\le2$, $n\ge1$. Then for a
  limit point $x$ of $(x_n)_n$ in $\Lambda_X(U)$ and a limit
  $2$-subspace $F_x$ of the sequence $F_{x_n}$ inside
  $E^{c,Y}_x$ (using the compactness of the Grassmannian
  over the compact set $\overline{U}$ together with the
  continuity of the splitting with respect to $Y_n$), we get
  $|\det(DX_T\vert_{F_x})|\le 2$. Since $T>0$ was
  arbitrarily chosen, this contradicts the assumption of
  sectional-expansion on $\Lambda$.

  Hence we may argue with any fixed $Y$ close enough to $X$
  in the $C^1$ topology.  Let us take $\Gamma$ a compact
  subset of $\Lambda^*_Y$. For $x\in\Gamma$ the uniform
  expansion along $N^u$ is obtained as follows.

  Let $v$ be a unit vector on $N^u_x$ and let $F_x$ be the
  subspace spanned by $v$ and $X(x)$.  For
  some $K>0$ we have $K^{-1}\le\|Y(x)\|\le K$ for all $x\in\Gamma$
  by compactness. Let us fix $t>0$ and consider the basis
  $\{\frac{T(x)}{\|T(x)\|},\, v\}$ of $F_x$.  We note that
  $DY_t(F_x)$ is a bidimensional subspace $F^t_x$ of
  $E^c_{Y_t(x)}$, where we take the basis
  $\{\frac{Y(Y_t(x))}{\|Y(Y_t(x)x)\|},\, w_t\}$, with
  \begin{align*}
    w_t:=\frac{\Pi_{Y_t(x)}\cdot DY_t(x)(
      v)}{\|\Pi_{Y_t(x)}\cdot DY^t(x)(v)\|}
    \quad\text{belonging to}\quad N^u_{Y_t(x)}.
  \end{align*}
  With respect to these orthonormal bases we have
  \begin{align*}
    DY_t\vert_{F_x} = \left[\begin{array}{cc}
        \frac{\|Y(Y_t(x))\|}{\|Y(x)\|} & \star  \\
        0 & \Delta
      \end{array} \right],
  \end{align*}
  because the flow direction is invariant.
  Hence
  \begin{align*}
    \det\big(DY_{t}\vert_{E^{c}_{x}} \big) =
    \frac{\|Y(Y_t(x))\|}{\|Y(x)\|} \Delta
    \le K^2 \Delta
  \end{align*}
  for some $K>0$ depending only on
  $\Gamma\subset\Lambda_Y(U)^*$, and
  \begin{align*}
    \|P^{\,Y,t}_x\cdot v\| & = \| \Pi_{X_t(x)}\cdot
    DY_{t}(x)(v) \| = \|\Delta \cdot w \| = |\Delta|
    \\
    & \ge K^{-2} |\det(DY_{t}\vert_{F_x}) | \ge K^{-2} C
    e^{\lambda t}.
  \end{align*}
  This proves that $N^u$ is uniformly expanded by the linear
  Poincar\'e flow $P^{Y,t}$ over $\Gamma$. Moreover, for
  every singularity $\sigma\in\Lambda$ we have $T_\sigma
  M=E^s_\sigma\oplus E^u_\sigma$ a sectional hyperbolic
  splitting, thus $\indi(\sigma)\ge\indi(\Lambda)$; in fact,
  sectional expansion on $E^c_\sigma$ ensures that either
  $\indi(\sigma)=\indi(\Lambda)$ or
  $\indi(\sigma)=\indi(\Lambda)+1$.

  Reciprocally, let us assume that $\Lambda$ is a
  compact attracting set with isolating neighborhood
  $U$ such that: the linear Poincar\'e flow over any
  compact subset $\Gamma\subset\Lambda^*_Y(U)$ is
  hyperbolic with constant index $\indi(\Lambda)$, for
  all $Y$ in a $C^1$ neighborhood $\U$ of $X$; and that
  the singularities $\sigma$ in $U$ for each $Y\in\U$
  are sectionally hyperbolic with index
  $\indi(\sigma)\ge\indi(\Lambda)$.  In particular, the
  index of all periodic orbits of $U$ for $Y\in\U$ is
  constant $\indi(\Lambda)$, and the flows in $\U$ are
  homogeneous. Hence, every periodic point $p$ in $U$
  for $Y$ is hyperbolic with uniform bounds on the
  expansion and contraction on the period and,
  moreover, admits a sectional-hyperbolic splitting
  $T_pM = {E}_p^{Y,s} \oplus E^{Y,c}_p$ of the
  tangent bundle with constant index $\indi(\Lambda)$
  and with angle between the stable and central
  directions uniformly bounded away from zero;
  see~\cite[Section 5.4.1]{AraPac2010}.

  This is enough to deduce that the tangent bundle on
  $\Lambda_Y(U)\cap\Omega(X)$ admits a partially
  hyperbolic splitting $E^s\oplus E^c$ with index $\dim
  E^s=\indi(\Lambda)$. Indeed, for a non-singular
  $x\in\Lambda_Y(U)$ we can use Shub's Closing Lemma to
  obtain sequences $Y^k\in\U$ and $p_k\in U$ a periodic
  point for $Y^k$ such that $Y^k\xrightarrow[]{C^1}Y$
  and $p_k\to x$. We then define $ E^{s,Y}_x = \lim_{k
    \rightarrow\infty}E^{s,Y^{k}}_{p_k}$ and
  $E^{c,Y}_x= \lim_{k \rightarrow
    \infty}E^{c,Y^{k}}_{p_{k}}.$ This decomposition
  will be $DY_t$-invariant and partially hyperbolic by
  construction.  Moreover the assumption on the index
  of the singularities ensures that the partial
  hyperbolic splitting on every periodic orbit for each
  flow $Y\in\U$ can be extended to a partially
  hyperbolic splitting on the entire $\Lambda$,
  including the singularities; see~\cite[Section
  5.4.2]{AraPac2010}. Having these properties robustly
  on $\U$ with sectional hyperbolicity on periodic
  orbits implies that the subbundle $E^c$ is
  sectionally expanding, for $Y\in\U$;
  see~\cite[Section 5.4.3]{AraPac2010}.  This completes
  the proof.
\end{proof}

Now Proposition~\ref{pr:J-monot-hyperb} shows that strict
$\J$-monotonicity for the linear Poincar\'e flow over a
compact invariant subset implies hyperbolicity. Together
with Theorem~\ref{thm:2secexp-LPFhyp} we conclude the proof
of sufficiency in Theorem~\ref{mthm:2-sec-exp-J-monot}.% , for

This completes the proof that strict $J$-monotonicity of the
linear Poincar\'e flow implies sectional hyperbolicity,
which is half of the statement of
Theorem~\ref{mthm:2-sec-exp-J-monot}.

\subsection{ Strict $\J$-monotonicity for
  the linear Poincar\'e flow}
\label{sec:hyperb-linear-poinca}

Now we prove that having a sectional hyperbolic splitting
implies that there exists a field of non-degenerate and
indefinite quadratic forms $\J$ with constant index equal to the
dimension of the contracting direction, such that the linear
Poincar\'e flow is strictly $\J_0$-monotonous on every
compact invariant set without singularities, for some
compatible field of forms $\J_0$, completing the proof of
Theorem~\ref{mthm:2-sec-exp-J-monot}.

Let $\Lambda=\Lambda_X(U)$ be a compact maximal invariant
set admitting a sectional hyperbolic splitting $T_\Lambda
M=E^s_\Lambda \oplus E^c_\Lambda$. As noted in the first
part of the proof of Theorem~\ref{thm:2secexp-LPFhyp}, the
existence of sectional hyperbolic splitting is a robust
property: there exists a neighborhood $\U$ of $X$ in the
$C^1$ topology in $\Mundo$ such that all $Y\in\U$ have a
maximal invariant subset $\Lambda_Y(U)$ which is also
sectional hyperbolic. Hence the results we obtain below hold
robustly in a neighborhood of $X$.

We have already shown, in
Section~\ref{sec:partial-hyperb-impli}, how to construct a
field of quadratic forms $\J$ such that $X$ is strictly
$\J$-separated on a neighborhood $V\subset U$ of $\Lambda$
satisfying for some $\lambda>0$ and all $x\in\Lambda$ and
$t>0$
\begin{align*}
  |DX_tv^+|=\sqrt{ \J(DX_tv^+)}\ge e^{\lambda t}
  \sqrt{\J(DX_tv^-)} = e^{\lambda t}|DX_tv^-|, \quad v^-\in
  E^s_x, v^+\in E^c_x, |v^\pm|=1.
\end{align*}
The results in \cite{Goum07} extend the properties of
adapted metrics to partial hyperbolic splittings, in such a
way that we can also obtain for all $t>0$
\begin{align*}
  |DX_tv^-|&\le e^{-\lambda t}, \quad v^-\in E^s_x, |v^-|=1.
\end{align*}
On $\Lambda^*=\Lambda\setminus \sing(X)$ we define
$N^s=\prod\cdot E^s$ and $N^u=\prod\cdot E^c$, where $\prod$
is the projection of the tangent bundle onto the
pseudo-orthogonal complement $N$ of $X$ with respect to
$\J$. We note that since $P^t=\prod\cdot DX_t$
\begin{align}
  |P^t\mid_{N^s_x}|\cdot|P^{-t}\mid_{N^u_{X_t(x)}}|
  &\le
  |DX_t\mid_{E^s_x}|\cdot|DX_{-t}\mid_{E^c_{X_t(x)}}|
  \le
  e^{-\lambda t},
  \quad\text{and} \label{eq:FLP-domin}
  \\
  |P^t\mid_{N^s_x}|
  &\le \label{eq:FLPs-contr}
  |DX_t\mid_{E^s_x}|
  \le  e^{-\lambda t},\quad
  x\in\Lambda^*, t>0
\end{align}
so that the linear Poincar\'e flow has a partially hyperbolic
splitting over $\Lambda^*$.

The assumption of sectional expansion ensures that, if we
fix any unit vector $v\in N^u_x$ for $x\in\Lambda^*$, then
for some $C, \lambda>0$ and every $t>0$
\begin{align*}
  Ce^{\lambda t}&\le
  |\det  DX_t\mid_{\gen\{X(x),v\}}|
  =
  \frac{\vol(DX_t v, X(X_t(x)))}{\vol(X(x),v)}
  \\
  &=
  \frac{|X(X_t(x))|}{|X(x)|}|DX_tv|\sin\angle(DX_tv,X(X_t(x)))
  =
  \frac{|X(X_t(x))|}{|X(x)|}|P^tv|.
\end{align*}
Since we are in a compact set we have
$0<c_0=\sup_{z\in\Lambda}|X(z)|<\infty$ and so
\begin{align}\label{eq:FLPc-exp}
  |P^tv|\ge\frac{C|X(x)|}{c_0} e^{\lambda t}, \quad
  x\in\Lambda^*, v\in N^u_x, |v|=1, t>0.
\end{align}
We write $c(x):=C|X(x)|/c_0$ and note that $0<c(x)\le1$ by
letting $t\to0$ in the above inequality.

We restrict now to the case of a compact invariant subset
$\Gamma$ of $\Lambda^*$. In this case $c(x)\ge c_1>0$ for
all $x\in\Gamma$ and $N^s_\Gamma\oplus N^u_\Gamma$ is a
uniformly hyperbolic splitting for $P^t$. We can then obtain
an adapted Riemannian metric for this splitting
following~\cite{Goum07} and define a field $\J_0$ of
quadratic forms using this adapted metric as in
Section~\ref{sec:constr-field-quadrat}. With respect to the
adapted metric we obtain both~(\ref{eq:FLP-domin})
and~(\ref{eq:FLPs-contr}), and also~(\ref{eq:FLPc-exp}) but
with unit constants multiplying the exponential.
From Remark~\ref{rmk:equivalent-qforms}, since $\J$ and
$\J_0$ have the same signs on the $E^s_\Gamma$ and
$E^c_\Gamma$, then $\J_0\sim\J$.

We show that $P^t$ is strictly $\J_0$-monotonous.  We
consider a vector $v\in N_x$ with $v=v^-+v^+, v^-\in N_x^s,
v^+\in N^u_x$ and $\J_0(v^+)-\J_0(v^-)=|v^+|^2+|v^-|^2=1$
for $x\in\Gamma$, and the norm of its image under the linear
Poincar\'e flow. Since $P^tv=P^tv^++P^tv^-$ and $N^s$ and
$N^u$ are $P^t$-invariant, if $v^+\neq\vec0$ we get for
$t>0$
\begin{align*}
  \J_0(P^tv)
  &=
  \J_0(P^tv^+) +\J_0(P^tv^-)
  =
  \J_0(P^tv^+)\cdot
  \left(1+\frac{\J_0(P^tv^-)}{\J_0(P^tv^+)}\right)
  \\
  &\ge
  e^{2\lambda t}\J_0(v^+)\left(1+e^{-2\lambda
          t}\frac{\J_0(v^-)}{\J_0(v^+)}\right),
\end{align*}
since $\J(P^tv^-)<0$. We note that the value of the left
hand side and the right hand side above are the same at
$t=0$. Moreover the derivative of the right hand side with
respect to $t$ at $t=0$ equals
\begin{align*}
  \left.\left[
  2\lambda e^{2\lambda t}\J_0(v^+)\left(1+e^{-2\lambda
          t}\frac{\J_0(v^-)}{\J_0(v^+)}\right)
      - e^{2\lambda t}\J_0(v^+)\cdot
      2\lambda e^{-2\lambda t}\frac{\J_0(v^-)}{\J_0(v^+)}
    \right]\right|_{t=0}
  &=
  2\lambda\J_0(v^+)>0.
\end{align*}
Hence we conclude that $
  \partial_t\J_0(P^tv)\mid_{t=0}
  \ge 2\lambda\J_0(v^+)>0
$
when $v$ has a non-zero positive component. In the remaining
case, $v=v^-$ we obtain (again because $\J_0(P^tv^-)<0$)
\begin{align*}
  \J_0(P^tv^-)\ge e^{-2\lambda t}\J_0(v^-)
\end{align*}
with the same value at $t=0$ on both sides, so that
$  \partial_t\J_0(P^tv^-)\mid_{t=0} \ge
  -2\lambda\J_0(v^-)>0$
also in this case.

We have proved that for all vector fields $Y$ sufficiently
$C^1$ close to $X$ and for every compact invariant subset
$\Gamma$ of $\Lambda_Y(U)^*$, we can find a field $\J_0$ of
quadratic forms compatible to $\J$ over $\Gamma$ such that
$P^t_Y$ is strictly $\J_0$-monotonous, as claimed in
Theorem~\ref{mthm:2-sec-exp-J-monot}.

This together with Theorem~\ref{thm:2secexp-LPFhyp}
completes the proof of Theorem~\ref{mthm:2-sec-exp-J-monot}.

% Now we observe that the Closing Lemma ensures that, for each
% non-wandering point $x$ of $X$, there are a vector field $Y$
% arbitrarily $C^1$ close to $X$ and a periodic orbit $\cO$ of
% $Y$ passing arbitrarily close to $x$. If
% $x\in\Lambda\cap\Omega(X)$, then the periodic orbit $\cO$ is
% contained in $\Lambda_Y(U)$ and $\cO$ is a compact invariant
% subset of $\Lambda_Y(U)^*$. This shows that for every
% $x\in\Lambda\cap\Omega(X)$ and $\epsilon>0$ we can find

\subsection{Criteria for $\J$-monotonicity of the linear
  Poincar\'e flow}
\label{sec:criter-j-monoton}

Here we prove Proposition~\ref{pr:P-J-monotonous}.  We have
already characterized partial hyperbolicity using the notion
of $\J$-separation, or the existence of infinitesimal
Lyapunov functions, which depend only on the vector field
$X$ and its derivative $DX$. To present a characterization
of sectional hyperbolicity along the same lines, we must use
the conclusion of Theorem \ref{mthm:2-sec-exp-J-monot}, and
obtain a criterion for the linear Poincar\'e flow to be
$\J$-monotonous.

The condition of $\J$-monotonicity for the linear Poincar\'e
flow can be expressed using only the vector field $X$ and
its space derivative $DX$.

Recall that a self-adjoint operator is said to be (positive)
non-negative if all eigenvalues are (positive) non-negative.

\begin{lemma}
  \label{le:P-J-monotonous}
  Let $X$ be a $\J$-non-negative (positive) vector field on
  $U$.  Then, the Linear Poincar\'e Flow is (strictly)
  $\J$-monotone if, and only if, the operator
  $$
  \hat J_x:= DX(x)^*\cdot \Pi_x^* J \Pi_x + \Pi_x^* J \Pi_x\cdot DX(x)
  $$
  is a non-negative (positive) self-adjoint operator.
\end{lemma}
Here, we consider $\Pi^*$ as the adjoint operator of the
orthogonal projection $\Pi$ in the definition of the linear
Poincar\'e flow.

The conditions above are again consequence of the corresponding
results for linear multiplicative cocycles over flows, as
explained in Section~\ref{sec:approach-via-linear}.

\begin{proof}
  We shall prove only the positive case, once the
  non-negative is similar.

  We denote by $|v|:=\langle J v,v\rangle^{1/2}$ the
  $\J$-norm of a vector $v$ and observe that we can write
  \begin{align*}
    \Pi_{X_t(x)} v:= v-\langle J
    v,\frac{X(X_t(x))}{|X(X_t(x))|} \rangle
    \frac{X(X_t(x))}{|X(X_t(x))|},
  \end{align*}
  for all $v\in T_xM$ with $ X(X_t(x))\neq\vec0$ and $t\ge0$.
  Then, to conclude that $\J(P^tv)>\J(v)$ for every $v\in
  N_x$, $x\in U, X(x)\neq\vec0$ it is enough to prove
  \begin{align}\label{eq:strict-Jmonotonous}
    \partial_t \J(P^tv)>0 \quad\text{for every}\quad v\in
    T_xM,\, X(X_t(x))\neq\vec0\qand t\ge0.
  \end{align}
  Reciprocally, if we have that $\J(P^tv)>\J(v)$ for every
  $v\in N_x$, $x\in U, X(x)\neq\vec0$, then we also must
  have (\ref{eq:strict-Jmonotonous}).

  Now the above derivative can be written, just like in the
  previous sections
  \begin{align}\label{eq:Jderivative1}
    \langle J \cdot
    \Pi_{X_t(x)}DX_tv,\partial_t(\Pi_{X_t(x)}DX_tv)\rangle +
    \langle J\cdot \partial_t(\Pi_{X_t(x)}DX_tv),
    \Pi_{X_t(x)}DX_tv \rangle.
  \end{align}
  To expand the above expression, we note that
  \begin{align*}
    \partial_t(\Pi_{X_t(x)}DX_tv) &=
    \partial_t\left(DX_tv - \langle J \cdot DX_tv,
      \frac{X(X_t(x))}{|X(X_t(x))|}\rangle
      \frac{X(X_t(x))}{|X(X_t(x))|} \right)
  \end{align*}
  can be written in the following way,
  \begin{align*}
    DX&(X_t(x))DX_tv + \langle J \cdot DX_tv,
    \frac{X(X_t(x))}{|X(X_t(x))|}\rangle\cdot
    \partial_t \frac{X(X_t(x))}{|X(X_t(x))|}
    \\
    &-\left( \langle J \cdot DX(X_t(x))DX_tv,
      \frac{X(X_t(x))}{|X(X_t(x))|}\rangle + \langle J \cdot
      DX_tv,
      \partial_t \frac{X(X_t(x))}{|X(X_t(x))|}\rangle
    \right) \frac{X(X_t(x))}{|X(X_t(x))|}.
  \end{align*}
  Since $\partial_t \frac{X(X_t(x))}{|X(X_t(x))|}$ equals
  \begin{align*}
     -\langle
    \frac{X(X_t(x))}{|X(X_t(x))|} ,
    DX(X_t(x))\frac{X(X_t(x))}{|X(X_t(x))|} \rangle \cdot
    \frac{X(X_t(x))}{|X(X_t(x))|} +
    DX(X_t(x))\frac{X(X_t(x))}{|X(X_t(x))|}
  \end{align*}
  and must be $\J$-orthogonal to the flow direction at
  $X_t(x)$, then this last expression is the projection on
  $N_{X_t(x)}$ as follows
  \begin{align*}
    \partial_t \frac{X(X_t(x))}{|X(X_t(x))|} = \left(
      \Pi_{X_t(x)}
      DX(X_t(x))\right)\cdot\frac{X(X_t(x))}{|X(X_t(x))|}
  \end{align*}
  Now replacing $X_t(x)$ by $z$ throughout and the vector
  $X(z)$ $\J$-normalized by $\hat X(z)$ we obtain the following
  expression for the derivative of $P^tv$
  \begin{align*}
    DX(z)&DX_tv - \langle J \cdot DX_tv, \hat X(z)\rangle
    \cdot \Pi_zDX(z)\hat X(z)
    \\
    &-\left( \langle J \cdot DX(z)DX_tv, \hat X(z)\rangle +
      \langle J \cdot DX_tv, \Pi_zDX(z)\hat X(z) \rangle
    \right) \hat X(z)
  \end{align*}
  or, easier for a geometrical interpretation
  \begin{align*}
    DX(z)&DX_tv - \langle J \cdot DX(z)DX_tv, \hat
    X(z)\rangle \hat X(z)
    \\
    &- \langle J \cdot DX_tv, \hat X(z)\rangle \cdot
    \Pi_zDX(z)\hat X(z) - \langle J \cdot DX_tv,
    \Pi_zDX(z)\hat X(z) \rangle \hat X(z).
  \end{align*}
  We observe that the first line above is the projection on
  $N_z$ of $DX(z)DX_tv$ so we have that $\partial_tP^{\,
    t}v$ equals
  \begin{align*}
     \Pi_zDX(z)DX_tv - \langle J \cdot
    DX_tv, \hat X(z)\rangle \Pi_zDX(z)\hat X(z) - \langle J
    \cdot DX_tv, \Pi_zDX(z)\hat X(z) \rangle \hat X(z).
  \end{align*}
  In the expression~\eqref{eq:Jderivative1}, we take the
  $\J$-(inner)-product with a vector on $N_z$, so the $\hat
  X$ component above contributes nothing to the final
  result. Therefore \eqref{eq:Jderivative1} becomes
  \begin{align*}
    \langle J &\cdot \Pi_{z}DX_tv , \Pi_zDX(z)DX_tv \rangle
    + \langle J \cdot \Pi_zDX(z)DX_tv , \Pi_{z}DX_tv \rangle
    \\
    &-\langle J \cdot DX_tv, \hat X(z)\rangle \big( \langle
    J \cdot \Pi_{z}DX_tv , \Pi_zDX(z)\hat X(z) \rangle +
    \langle J \cdot \Pi_zDX(z)\hat X(z) , \Pi_{z}DX_tv
    \rangle \big)
  \end{align*}
  and using the adjoint of $DX(z)=DX(X_t(x))$ we define
  $\hat J = \hat J_x := \big(\Pi_x DX(x)\big)^* J \Pi_x +
  \Pi_x^* J \Pi_x DX(x)$ and obtain
  \begin{align*}
    \partial_tP^{\, t}v = \langle \hat J_{X_t(x)} DX_tv,
    DX_tv\rangle - \langle J \cdot DX_tv, \hat
    X(X_t(x))\rangle\cdot \langle \hat J \cdot DX_tv ,\hat
    X(X_t(x)) \rangle.
  \end{align*}
  Letting $t=0$, since $ \langle J \cdot v, \hat
  X(x)\rangle=0$ for $v\in N_x$, we arrive at
  \begin{align}
    \label{eq:LPF-J-derivative}
    \partial_t\big(P^{\, t}v\big)\mid_{t=0} = \langle \hat
    J_{x} v, v\rangle - \langle J \cdot v, \hat
    X(x)\rangle\cdot \langle \hat J \cdot v ,\hat X(x)
    \rangle = \langle \hat J_{x} v, v\rangle.
  \end{align}
  We conclude that condition \eqref{eq:strict-Jmonotonous}
  is equivalent to
  \begin{align}
    \label{eq:strict-J-monotonous-DX}
    \langle \hat J_{X_t(x)} v, v\rangle >0, \quad v\in
    N_{X_t(x)},
  \end{align}
  that is, \emph{$\hat J_x$ is a positive definite
    self-adjoint operator on $N_x$ for each $x\in U$ with
    $X(x)\neq\vec0$.}  Indeed, by the flow property of
  $P^{\,t}$ we have, for all $s>0$
  \begin{align*}
    \partial_t\J(P^{\,t+s}v)\mid_{t=0}=
    \partial_t\J\big(P^{\,t}(P^{\,s}v)\big)\mid_{t=0}>0
    \quad\text{because}\quad P^{\,s}v\in N_{X^s(x)}
  \end{align*}
  and $P^{\, s}:N_x\to N_{X^s(x)}$ is an isomorphism.
\end{proof}

So \eqref{eq:strict-J-monotonous-DX} is the condition that
the vector field and its derivative must satisfy in $U$,
except at singularities, so that the Linear Poincar\'e Flow
admits a hyperbolic splitting.

This concludes the proof of
Proposition~\ref{pr:P-J-monotonous}.

\def\cprime{$'$}

\bibliographystyle{abbrv}
%\bibliography{../../Trabalho/bibliobase/bibliography}

\end{document}